\documentclass[reqno,11pt]{amsart}
\usepackage{amsmath,amssymb,dsfont}
\usepackage[numbers]{natbib}
\usepackage[a4paper, margin = 3.1cm]{geometry}
\newtheorem{theorem}{Theorem}[section]
\newtheorem{corollary}[theorem]{Corollary}
\newtheorem{lemma}[theorem]{Lemma}
\newtheorem{proposition}[theorem]{Proposition}
\theoremstyle{definition}
\newtheorem{definition}[theorem]{Definition}
\theoremstyle{remark} \theoremstyle{remark}

\numberwithin{equation}{section}
\DeclareMathOperator*{\esssup}{\mathrm{ess\,sup}}

\allowdisplaybreaks[3]

\newcommand{\Rd}{\mathds{R}^d}
\newcommand{\intRd}{\int\limits_{\mathds{R}^d}}
\newcommand{\intGamN}{\int\limits_{\Gamma_0}}
\newcommand{\suml}[1]{\sum\limits_{#1}}
\newcommand{\prodl}[1]{\prod\limits_{#1}}
\newcommand{\intl}[1]{\int\limits_{#1}}
\newcommand{\bs}{\backslash}
\newcommand{\cOne}{\langle c_1 \rangle}
\newcommand{\cTwo}{\langle c_2 \rangle}
\newcommand{\coMax}{c_1^{\text{max}}}

\newcommand{\phinorm}{\langle \phi \rangle}

\newcommand{\bbs}{\mathcal{B}_{bs}}

\newcommand{\Usigma}{\mathcal{U}^{\sigma}}
\newcommand{\Pexp}{\mathcal{P}_{\exp}}

\title[Random jumps and coalescence in the continuum]{Random jumps and coalescence in the continuum: evolution of states of an infinite system}

\author{Yuri Kozitsky}
\address{Instytut Matematyki, Uniwersytet Marii Curie-Sk{\l}odowskiej, 20-031 Lublin, Poland}
\email{jkozi@hektor.umcs.lublin.pl}

\author{Krzysztof Pilorz}
\address{Instytut Matematyki, Uniwersytet Marii Curie-Sk{\l}odowskiej, 20-031 Lublin, Poland}
\email{krzysztof.pilorz@poczta.umcs.lublin.pl}

\begin{document}


\begin{abstract}

The dynamics of an infinite continuum system of randomly jumping and
coalescing point particles is studied. The states of the system are
probability measures on the corresponding configuration space
$\Gamma$ the evolution of which is constructed in the following way.
The evolution of observables $F_0\to F_t$ is obtained from a
Kolmogorov-type evolution equation. Then the
evolution of states $\mu_0\to \mu_t$ is defined by the relation
$\mu_0(F_t) =\mu_t(F_0)$ for $F_0$ belonging to a measure-defining
class of functions. The main result of the paper is the proof of the
existence of the evolution of this type for a bounded time horizon.
\end{abstract}
\maketitle

\section{Introduction}

The random motion of infinite systems in the course of which the
constituents can merge attracts considerable
attention. The Arratia flow introduced in \cite{Arrat} provides an
example of this kind. In recent years, it has been being extensively
studied, see \cite{Beres,Kovn,KovnR,LeJ} and the references therein.

In Arratia's model, an infinite number of Brownian particles move in
$\mathds{R}$ independently up to their collision, then merge and
move together as single particles. Correspondingly, the description
of this motion (and its modifications) is performed in terms of
stochastic (diffusion) processes. In this work, we propose an
alternative look at this kind of motion. Similarly as in the
Kawasaki model \cite{Kawasaki2, Kawasaki}, in the model which we
propose and study here point particles perform random jumps with
repulsion in $\mathds{R}^d$, $d\geq 1$. Additionally, two particles
(located at $x$ and $y$) can merge into a particle (located at $z$)
with intensity (probability per time) $c_1(x,y;z)$. Thereafter, this
new particle participates in the motion. The phase space of such a
system is the set $\Gamma$ of all locally finite configurations
$\gamma \subset \mathds{R}^d$, see \cite{Kawasaki2, Kawasaki, FKKK,
SpatialEcologicalModel, HarmonicAnalysis}, and the states of the
system are probability measures on $\Gamma$ the set of which will be
denoted as $\mathcal{P}(\Gamma)$. The description of their evolution
$\mu_0 \to \mu_t$ is based on the relation $\mu_t (F_0) = \mu_0
(F_t)$ where $F_0:\Gamma \to \mathds{R}$ is supposed to belong to a
measure-defining class of functions, $\mu(F) := \int F d \mu$ and
the evolution $F_0\to F_t$ is obtained by solving the Kolmogorov
equation
\begin{equation}\label{KE}
    \frac{d}{dt} F_t = LF_t, \qquad F_t|_{t=0} = F_0,
\end{equation}
in which the operator $L$ specifies the model, see \eqref{OperatorL}
below. The main result of this work (Theorem \ref{Th2}) is the
construction of the evolution of this kind for $t<T$ (with some
$T<\infty$) and $\mu_0$ belonging to a certain set of probability
measures on $\Gamma$. The basic aspects of this construction can be
outlined as follows. Let $\varOmega$ stand for the set of all
compactly supported continuous functions $\omega :\mathds{R}^d\to
(-1,0]$. Set
\begin{equation}
  \label{Co1}
  F^\omega (\gamma) = \prod_{x\in \gamma} (1+\omega (x)), \qquad \omega \in \varOmega.
\end{equation}
Then the collection $\{F^\omega:\omega \in \varOmega\}$ is a
measure-defining class. The set of measures $\mathcal{P}_{\rm
exp}\subset \mathcal{P}(\Gamma)$ with which we will work is
defined by the condition that its members enjoy the following
property: the map $\varOmega \ni \omega \mapsto \mu(F^\omega)\in
\mathds{R}$ can be continued to an exponential type entire function
on $L^1(\mathds{R}^d)$. Then, for $\mu\in \mathcal{P}_{\rm exp}$, we
set $B_\mu(\omega) = \mu (F^\omega)$ and derive $\widetilde{L}$ from
$L$ according to the rule $(\widetilde{L}B_\mu)(\omega) =
\mu(LF^\omega)$. Thereafter, we construct the evolution
$B_{\mu_0}\to B_t$ by solving the corresponding evolution equation.
The next (and the hardest) part of this scheme is to prove that $B_t
= B_{\mu_t}$ for a unique $\mu_t \in \mathcal{P}_{\rm exp}$.

In Section
\ref{Section:Preliminaries}, we outline the mathematical
background of the paper. In Section \ref{Section:Results}, we
introduce the model and present the results in the form of Theorems
\ref{Th1} and \ref{Th2}. Their proof is performed in Sections
\ref{Section:ProofTh1} and \ref{Section:ProofTh2}, respectively.

\section{Preliminaries}\label{Section:Preliminaries}
As mentioned above, we work with the phase space
\[
    \Gamma = \{\gamma \subset \Rd: |\Lambda \cap \gamma| < \infty \text{ for any compact } \Lambda \subset \Rd\},
\]
where $|\cdot|$ denotes cardinality. It is
equipped with the vague topology (see e.g., \cite{FKKK,HarmonicAnalysis}) and the
corresponding Borel $\sigma$-field $\mathcal{B}(\Gamma)$.  The set of all finite
configurations is denoted by $\Gamma_0$. It is the union of the sets $\Gamma^{(n)} = \{\gamma \in \Gamma: |\gamma| = n\}$, 
$n \in \mathds{N}_0$
that allows one to endow $\Gamma_0$ with the disjoint union topology
and with the corresponding Borel $\sigma$-field
$\mathcal{B}(\Gamma_0)$. The topology of each $\Gamma^{(n)}$ is
obtained from the Euclidean topology of $(\Rd)^n$ by the
symmetrization. Note that $\Gamma_0 \in \mathcal{B}(\Gamma)$.

It can be shown, cf. \cite{FKKK}, that a function $G: \Gamma_0
\rightarrow \mathds{R}$ is measurable if and only if there exists a
collection of symmetric Borel functions $G^{(n)}:(\Rd)^n \rightarrow
\mathds{R}$ such that, for any $n \in \mathds{N}$,
\begin{equation}\label{GamNRdFun}
    G(\eta) = G^{(n)}(x_1, \ldots, x_n), \qquad \eta = \{x_1, \ldots, x_n\}.
\end{equation}
\begin{definition}
  \label{Gdf}
A function $G:\Gamma_0\to \mathds{R}$ is said to have bounded
support if there exist a compact set $\Lambda \subset \Rd$ (spatial
support) and an integer $N \in \mathds{N}$ (quantitative bound) such
that $G(\eta) = 0$ whenever $\eta \cap \Lambda^c \neq \emptyset$ or
$|\eta| > N$. By $\bbs$ we denote the set of all bounded measurable
functions of bounded support.
\end{definition}
The Lebesgue-Poisson measure $\lambda$ on $\Gamma_0$ is defined by
the integrals
\begin{equation}\label{LPDef}
\intGamN G(\eta) \lambda (d\eta) = G(\emptyset) + \suml{n=1}^\infty
\frac{1}{n!} \intl{(\Rd)^n} G^{(n)}(x_1, x_2, \ldots, x_n) dx_1 dx_2
\ldots dx_n,
\end{equation}
holding for all $G \in \bbs$. For a measurable set $\Lambda \subset
\Rd$, we define $\Gamma_{\Lambda} = \{ \gamma \in \Gamma: \gamma \subset \Lambda\}$.
Clearly, $\Gamma_{\Lambda}\in \mathcal{B}(\Gamma)$. We endow
$\Gamma_{\Lambda}$ with the topology induced from the vague topology
of $\Gamma$, so that its Borel $\sigma$-field is
$\mathcal{B}(\Gamma_{\Lambda}) = \{A \cap \Gamma_{\Lambda}: A \in
\mathcal{B}(\Gamma)\}$. For a given measure $\mu\in
\mathcal{P}(\Gamma)$, we define its projection $\mu^{\Lambda}$  by
\begin{equation}\label{Projection}
    \mu^{\Lambda} (A) = \mu(p_{\Lambda}^{-1}(A)), \qquad A \in \mathcal{B}(\Gamma_\Lambda), \qquad p_\Lambda (\gamma) := \gamma\cap \Lambda.
\end{equation}
For each $\mu\in \mathcal{P}(\Gamma)$, we can set, see (\ref{Co1}),
\begin{equation}
  \label{o1}
  B_\mu (\omega) = \mu(F^\omega):= \int_{\Gamma}F^\omega (\gamma) \mu
  (d\gamma), \qquad \omega \in
  \varOmega.
\end{equation}
The collection $\{F^\omega: \omega \in \varOmega\}$ is a
measure-defining class in the sense that
$\mu(F^\omega)=\nu(F^\omega)$ holding for all $\omega \in \varOmega$
implies $\mu=\nu$ for each $\mu,\nu \in \mathcal{P}(\Gamma)$, see
\cite[page 426]{Kawasaki2}. Then the action of $L$ can be
transferred to $B_\mu$ by means of the rule
\begin{equation}
  \label{o2}
  (\widetilde{L}B_\mu)(\omega) = \mu(LF^\omega).
\end{equation}
This allows one to pass from \eqref{KE} to the following evolution equation
\begin{equation}
\label{o3} \frac{d}{dt} B_t = \widetilde{L}B_t, \qquad B_t|_{t=0} =
B_{\mu_0}, \quad \mu_0 \in \mathcal{P}_{\rm exp}.
\end{equation}
The advantage of using $\mathcal{P}_{\rm exp}$ is that, for each of
its members, the function $B_\mu$ admits the representation
\begin{eqnarray}
  \label{o4}
B_\mu (\omega) & = & 1 + \sum_{n=1}^\infty \frac{1}{n!}
\int_{(\mathds{R}^d)^n} k^{(n)}_\mu (x_1, \dots , x_n) \omega(x_1)
\cdots \omega(x_n)d x_1 \cdots d x_n \\[.2cm] \nonumber & = &
\int_{\Gamma_0} k_\mu (\eta)e(\omega; \eta)\lambda(d\eta).
\end{eqnarray}
Here every $k^{(n)}_\mu$ is a symmetric element of $L^\infty
((\mathds{R}^d)^n)$ such that
\begin{equation}
  \label{o5}
  \|k^{(n)}_\mu\|_{L^\infty
((\mathds{R}^d)^n)} \leq C^n, \qquad n \in \mathds{N},
\end{equation}
with one and the same $C>0$ for all $n\in \mathds{N}$. In the second line of (\ref{o4}), we
use the measure $\lambda$ introduced in (\ref{LPDef}), $k_\mu :
\Gamma_0 \to \mathds{R}$ is defined by $k_\mu (\eta) =
k^{(n)}_\mu (x_1, \dots , x_n))$ for $\eta = \{x_1, \dots , x_n\}$,
cf. (\ref{GamNRdFun}), and
\begin{equation*}
e(\omega; \eta) := \prod_{x\in \eta} \omega (x), \qquad \eta \in
\Gamma_0.
\end{equation*}
The function $k_\mu$ is called the correlation function of the state
$\mu$, whereas $k^{(n)}_\mu$ is its $n$-th order correlation
function. $k_\mu$ completely characterizes $\mu\in \mathcal{P}_{\rm
exp}$. For instance, $k_{\pi_\varrho}(\eta) = e(\varrho;\eta)$ for
the Poisson measure $\pi_\varrho$ with density
$\varrho:\mathds{R}^d\to [0,+\infty)$. On the other hand, the
following is known, see \cite[Theorems 6.1, 6.2 and Remark
6.3]{HarmonicAnalysis}.
\begin{proposition}\label{CorFunProp}
A function $k: \Gamma_0 \rightarrow \mathds{R}$ is a correlation
function of a unique measure $\mu \in \Pexp$ if and only if it
satisfies the conditions: (a) $k(\emptyset) = 1$; (b) the estimate
in (\ref{o5}) holds for some $C>0$ and all $n\in \mathds{N}$; (c)
for each $G\in \bbs^*$, the following holds
\begin{equation}
  \label{o7}
  \langle\! \langle G,k \rangle\!\rangle := \int_{\Gamma_0} G(\eta)
  k(\eta) \lambda(d \eta) \geq 0.
\end{equation}
\end{proposition}
Here
\begin{equation} \label{o8}
\bbs^*:= \{ G\in \bbs: (K G) (\eta) \geq 0\}, \qquad (KG)(\eta) :=
\sum_{\xi\subset \eta} G(\xi).
\end{equation}
Notably, the cone $\{G\in \bbs: G(\eta ) \geq 0\}$ is a proper
subset of $\bbs^*$.
\begin{corollary}
  \label{oco}
An exponential type entire function $B:L^1(\mathds{R})\to
\mathds{R}$ satisfies (\ref{o1}) for a unique $\mu\in
\mathcal{P}_{\rm exp}$ if and only if it admits the expansion as in
(\ref{o4}) with $k$ satisfying the conditions of Proposition
\ref{CorFunProp}.
\end{corollary}
Having in mind the latter facts we will look for the solutions of
(\ref{o3}) in the form
\begin{equation}
\label{O} B_t (\omega) = \langle \!\langle e(\omega;\cdot) ,
k_t\rangle \!\rangle
\end{equation}
with $k_t$ satisfying
\begin{equation}
  \label{CFE}
  \frac{d}{dt} k_t = L^\Delta k_t, \qquad k_t|_{t=0}= k_{\mu_0},
\end{equation}
where $L^\Delta$ is to be obtained from $\widetilde{L}$ (and thus from $L$) according to the rule, cf. (\ref{o2}) and
(\ref{O}),
\begin{equation}
  \label{o9}
(\widetilde{L}B_t)(\omega) = \langle \!\langle e(\omega;\cdot) ,
L^\Delta k_t \rangle \!\rangle
\end{equation}
For $\mu \in \Pexp$ and a compact $\Lambda \subset \Rd$, the projection of $\mu$ defined in \eqref{Projection}
is absolutely continuous with respect to the Lebesgue-Poisson measure $\lambda$. Let $R_\mu^\Lambda$ be its Radon-Nikodym derivative.
It is related to the correlation function $k_\mu$ by
\begin{equation}\label{CorAndDens}
    k_\mu(\eta) = \int\limits_{\Gamma_\Lambda} R_\mu^\Lambda(\eta \cup \xi) \lambda(d\xi), \quad \eta \in \Gamma_\Lambda.
\end{equation}
One of our tools in this work is based on the Minlos lemma
according to which, cf. \cite[eq. (2.2)]{FKKK},
\begin{equation*}
 \intGamN \intGamN G(\eta \cup \xi) H (\eta, \xi) \lambda(d\eta) \lambda(d\xi) =
 \intGamN G(\eta) \suml{\xi \subset \eta} H(\xi, \eta \bs \xi) \lambda(d\eta),
\end{equation*}
holding for appropriate  $G, H: \Gamma_0 \rightarrow \mathds{R}$.
By taking here
\[
    H(\eta_1, \eta_2) = \left\{ \begin{array}{ll} h(x, \eta_2), & \ \eta_1 = \{x\} \\ 0, & |\eta_1| \neq 1 \end{array} \right.
\]
and then by (\ref{LPDef}) we obtain its following special case
\begin{equation}\label{Minlos1}
\intGamN \intRd G(\eta \cup x) h(x,\eta) dx \lambda(d\eta) = \intGamN \suml{x \in \eta} G(\eta) h(x, \eta \bs x) \lambda(d \eta).
\end{equation}
Analogously, for
\[
H(\eta_1, \eta_2, \eta_3) = \left\{ \begin{array}{ll} h(x, y,
\eta_3), & \ \eta_1 = \{x\}, \eta_2 = \{y\} \\[.2cm] 0, & |\eta_1| \neq 1
\text{ or } |\eta_2| \neq 1, \end{array}  \right.
\]
one gets
\begin{eqnarray}\label{Minlos2}
& &\frac{1}{2} \intGamN \intRd \intRd G(\eta \cup \{x,y\}) h(x, y, \eta) dx dy \lambda(d\eta) \nonumber \\
& & \quad = \intGamN \suml{\{x, y\} \subset \eta} G(\eta) h(x, y, \eta \bs \{x, y\}) \lambda(d\eta).
\end{eqnarray}

\section{The Results}\label{Section:Results}

Our model is specified by the operator $L$ the action of which on an
observable $F: \Gamma \rightarrow \mathds{R}$ is 
\begin{gather}\label{OperatorL}
    (LF)(\gamma) = \suml{\{x,y\} \subset \gamma} \  \intRd c_1(x,y;z) \Big(F\big(\gamma \bs \{x,y\} \cup z \big) - F(\gamma) \Big) dz
    \\[.2cm]
    + \suml{x \in \gamma} \  \intRd \tilde{c}_2(x,y;\gamma) \Big(F\big(\gamma \bs x \cup y\big) - F(\gamma) \Big) dy. \nonumber
\end{gather}
Here $c_1\geq 0$ is the intensity of the coalescence of the
particles located at $x$ and $y$ into a new particle located at $z$.
Note that $c_1$ does not depend on the elements of $\gamma$ other
than $x$ and $y$. For simplicity, we assume that $c_1(x,y;z) =
c_1(y,x;z)= c_1 (x+u, y+u; z+u)$ for all $u\in \mathds{R}^d$. For a
more general version of this model, see \cite{RepulsiveCoalescence}.
The second summand in (\ref{OperatorL}) describes jumps performed by
the particles. As in \cite{Kawasaki2,Kawasaki}, we set
\[
    \tilde{c}_2(x,y;\gamma) = c_2(x-y) \prodl{u \in \gamma \bs x} e^{-\phi(y-u)} ,
\]
with $\phi$ and $c_2$ being the repulsion potential and the jump
kernel, respectively. By these assumptions the model is translation
invariant. The functions $c_1$, $c_2$ and $\phi$ take non-negative
values and satisfy the following conditions:
\begin{gather} \label{ModelAssumptions}
    \int\limits_{(\mathds{R}^d)^2} c_1(x_1,x_2;x_3) dx_i dx_j = \cOne \ < \infty,
    \\[.2cm]
    \coMax := \sup_{x, y \in \Rd} \intRd c_1(x,y;z) dz < \infty \nonumber,
    \\[.2cm]
    \cTwo \ :=\intRd  c_2(x) dx   < \infty , \quad
    \phinorm := \intRd \phi(x) dx < \infty \nonumber,
    \\[.2cm]
    |\phi| := \sup_{x \in \Rd} \phi(x) < \infty \nonumber.
\end{gather}
Now we pass to the equation in \eqref{CFE}. The corresponding operator $L^\Delta$ is to be calculated from (\ref{OperatorL})
by (\ref{o2}) and (\ref{o9}). It thus takes the form, cf. \cite{RepulsiveCoalescence},
\begin{equation*}
    L^\Delta = L_1^\Delta + L_2^\Delta,
\end{equation*}
where $L_1^\Delta = L_{11}^\Delta + L_{12}^\Delta + L_{13}^\Delta + L_{14}^\Delta$ is the part responsible for
the coalescence whereas $L_2^\Delta = L_{21}^\Delta + L_{22}^\Delta$ describes the jumps. Their summands are:
\begin{gather}
\label{O1}
    (L^{\Delta}_{11} k) (\eta) = \ \ \frac{1}{2} \int\limits_{(\mathds{R}^d)^2} \sum\limits_{z \in \eta} c_1(x,y;z) k(\eta \backslash z \cup \{x,y\})  dx dy,
    \\[.2cm] \nonumber
    (L^{\Delta}_{12} k) (\eta) = - \frac{1}{2} \int\limits_{(\mathds{R}^d)^2} \sum\limits_{x \in \eta} c_1(x,y;z) k(\eta \cup y)  dy dz,
    \\[.2cm] \nonumber
    (L^{\Delta}_{13} k )(\eta) = - \frac{1}{2} \int\limits_{(\mathds{R}^d)^2} \sum\limits_{y \in \eta} c_1(x,y;z) k(\eta \cup x)  dx dz,
    \\[.2cm] \nonumber
    (L^{\Delta}_{14} k )(\eta) =  - \Psi (\eta) k(\eta), \qquad \Psi(\eta):= \int\limits_{\mathds{R}^d}
    \sum\limits_{\{x,y\} \subset \eta} c_1(x,y;z) dz,
\end{gather}
and
\begin{align*}
    L^{\Delta}_{21} k(\eta) &= \intRd \sum\limits_{y \in \eta}  c_2(x-y) \prodl{u \in \eta \backslash y} e^{-\phi(y-u)} \ (Q_y k)(\eta \backslash y \cup x) dx,
    \\[.2cm]
    L^{\Delta}_{22} k(\eta) &= - \intRd \sum\limits_{x \in \eta} c_2(x-y) \prodl{u \in \eta \backslash x} e^{-\phi(y-u)} \ (Q_y k)(\eta) dy,
\end{align*}
where
\begin{align} \label{QyNotion}
    (Q_y k)(\eta) = \intGamN k (\eta \cup \xi) \prodl{u \in \xi}  (e^{-\phi(y-u)} - 1) \lambda (d\xi).
\end{align}
In view of (\ref{o5}), the Banach spaces for \eqref{CFE} ought to be
of $L^\infty$ type. Thus, we set
\begin{equation}\label{KSpace}
    \mathcal{K}_\theta = \left\{k: \Gamma_0 \rightarrow \mathds{R}: \|k\|_\theta < \infty \right\}, \quad \theta \in \mathds{R},
\end{equation}
with
\[
    \|k\|_\theta = \esssup_{\eta \in \Gamma_0} \left(|k(\eta)| e^{- \theta |\eta|} \right) = \sup_{n\geq 0}\left( e^{-\theta n}
    \|k^{(n)}\|_{L^\infty((\mathds{R}^n))}\right).
\]
By this definition it follows that each $k \in \mathcal{K}_\theta$
satisfies
\begin{equation}\label{KNormEstimate}
    |k(\eta)| \leq e^{\theta |\eta |} \|k\|_\theta.
\end{equation}
With the help of this estimate and (\ref{O1}), we show that the first three summands in
$L^\Delta_1$ satisfy
\begin{equation}
 \label{O2}
 \left|(L^\Delta_{1i}k)(\eta)\right| \leq \left(\frac{1}{2} \langle c_1 \rangle e^\theta \|k\|_\theta \right) |\eta| e^{\theta |\eta|}, \quad i=1,2,3.
\end{equation}
At the same time, $\Psi(\eta) \leq c_1^{\rm max} |\eta|(|\eta|-1)/2$, which yields the following estimate
\begin{equation}
 \label{O3}
 \left|(L^\Delta_{14}k)(\eta)\right| \leq \left(\frac{1}{2}  c_1^{\rm max}  \|k\|_\theta \right) |\eta|(|\eta|-1)e^{\theta |\eta|}.
\end{equation}
Since $L^\Delta_2$ coincides with the corresponding operator of the
Kawasaki model, by \cite[eq. (3.18)]{Kawasaki2} we have
\begin{equation}
 \label{O4}
\left|(L^\Delta_{2}k)(\eta)\right| \leq \left(2 \langle c_2 \rangle \exp\left(\langle \phi \rangle e^\theta \right)\|k\|_\theta \right)|\eta| e^{\theta |\eta|}.
\end{equation}
Let us now define $L^\Delta$ in a given $\mathcal{K}_\theta$. To this end, we set
\begin{equation}
 \label{o10}
\mathcal{D}_\theta = \{ k \in \mathcal{K}_\theta: \exists C_k>0 \ |\eta|^2 |k(\eta)| \leq C_k e^{\theta |\eta|}\}, \qquad \theta \in \mathds{R}.
\end{equation}
Then, similarly as in (\ref{O2}) -- (\ref{O4}), we obtain that both
$L^\Delta_1$ and $L^\Delta_2$ map the elements of
$\mathcal{D}_\theta$ into $\mathcal{K}_\theta$. Let
$L^\Delta_\theta$ denote the operator $(L^\Delta,
\mathcal{D}_\theta)$. Then, in the Banach space
$\mathcal{K}_\theta$, the problem in (\ref{CFE}) takes the form
\begin{equation}
 \label{OCFE}
 \frac{d}{dt} k_t = L^\Delta_\theta k_t, \qquad k_{t}|_{t=0} = k_0.
\end{equation}
\begin{definition}
 \label{Odf}
A classical solution of (\ref{OCFE}) on a given time interval
$[0,T)$ is a continuous function $[0,T)\ni t \mapsto k_t \in
\mathcal{D}_\theta$ that is continuously differentiable in
$\mathcal{K}_\theta$ on $(0,T)$ and is such that both equalities in
(\ref{OCFE}) hold.
\end{definition}
As is typical for problems like in (\ref{OCFE}), in view of the
complex character of the corresponding operator it might be
unrealistic to expect the existence of classical solutions for all
possible $k_0\in \mathcal{D}_\theta$. Thus, we will restrict the
choice of $k_0$ to a proper subset of the domain (\ref{o10}). For
$\theta' > \theta$, we have that $\mathcal{K}_\theta \hookrightarrow
\mathcal{K}_{\theta'}$, i.e., $\mathcal{K}_\theta$ is continuously
embedded in $\mathcal{K}_{\theta'}$. Similarly as in
\cite{Kawasaki2,Kawasaki,FKKK} we will solve (\ref{OCFE}) in the
scale $\{ \mathcal{K}_\theta\}_{\theta\in \mathds{R}}$. By means of
the estimates in (\ref{O2}) -- (\ref{O4}) one concludes that
$L^\Delta$ can be defined as a bounded linear operator from
$\mathcal{K}_{\theta}$ to $\mathcal{K}_{\theta'}$ whenever $\theta'
> \theta$. We shall denote this operator by $L^\Delta_{\theta'\theta}$. By this estimate one also gets that
\begin{equation}
 \label{o11}
 \mathcal{K}_{\theta} \subset \mathcal{D}_{\theta'}, \qquad
 \theta'>\theta,
\end{equation}
and
\begin{equation}
  \label{O5}
  L^\Delta_{\theta'} k =  L^\Delta_{\theta' \theta} k, \qquad k\in
  \mathcal{K}_\theta.
\end{equation}
Set
\begin{gather}
  \label{o26}
  \beta (\theta)  =  \frac{3}{2}\langle c_1 \rangle e^\theta + 2\exp\left( \langle \phi\rangle e^\theta\right) \langle c_2
  \rangle, \\[.2cm] \nonumber
    T(\theta', \theta)  =  \frac{\theta' - \theta}{\beta(\theta')}, \qquad \theta'>\theta.
\end{gather}
\begin{theorem}\label{Th1}
For each $\alpha_0 \in \mathds{R}$ and $\alpha_* > \alpha_0$, and
for an arbitrary $k_0 \in \mathcal{K}_{\alpha_0}$, the problem in
\eqref{OCFE} has a unique classical solution $k_t \in
\mathcal{K}_{\alpha_*}$ on $[0, T (\alpha_*, \alpha_0))$.
\end{theorem}

A priori the solution $k_t$ described in Theorem \ref{Th1} need not
be a correlation function of any state, which means that the result
stated therein has no direct relation to the evolution of states of
the system considered. Our next result removes this drawback.
\begin{theorem}\label{Th2}
Let $\mu_0 \in \Pexp$ be such that $k_{\mu_0}\in
\mathcal{K}_{\alpha_0}$. Then, for each $\alpha_*>\alpha_0$, the
evolution $k_{\mu_0}\to k_t$ described in Theorem \ref{Th1} has the
property: for each $t< T(\alpha_*,\alpha_0)/2$, $k_t$ is the
correlation function of a unique state $\mu_t \in \Pexp$.
\end{theorem}
By Theorem \ref{Th2} we  also have the evolution $B_{\mu_0}
\to B_t= B_{\mu_t} = \langle\!\langle e(\cdot;\cdot),
k_t\rangle\!\rangle$, where $B_t$ solves (\ref{o3}), cf. (\ref{O}).
Along with its purely theoretical value, this result may serve as a
starting point for a numerical study of the random motion of this
type, cf. \cite{Omel}, including its consideration at different space
and time scales \cite{Banas,Presutti}. To this end one can use kinetic equations related 
to the model specified in (\ref{OperatorL}), 
see \cite{RepulsiveCoalescence}.

\section{Proof of theorem \ref{Th1}}\label{Section:ProofTh1}

The solution in question will be obtained in the form
\begin{equation}
  \label{SiQ}
 k_t = Q_{\alpha_*\alpha_0}(t) k_0,
\end{equation}
where the family of bounded operators $Q_{\alpha_*\alpha_0}(t):
\mathcal{K}_{\alpha_0} \to \mathcal{D}_{\alpha_*} \subset
\mathcal{K}_{\alpha_*}$, $t\in(0, T(\alpha_*, \alpha_0))$ satisfies 
\begin{equation}
  \label{S1}
  \frac{d}{dt} Q_{\alpha_*\alpha_0}(t) = L^\Delta_{\alpha_*}
  Q_{\alpha_*\alpha_0}(t).
\end{equation}
Additionally, $Q_{\alpha_*\alpha_0}(0)$ is considered as the embedding operator, and 
hence $k_t$ given in (\ref{SiQ}) satisfies the initial condition up to this embeding.
Each $Q_{\alpha_*\alpha_0}(t)$ is constructed as a series of
$t$-dependent operators, convergent in the operator norm topology for $t<
T(\alpha_*, \alpha_0)$. In estimating the norms of these operators
we crucially use (\ref{O2}) -- (\ref{O4}).

As the right-hand sides of (\ref{O2}) and (\ref{O3}) contain
different powers of $|\eta|$, it is convenient to split $L^\Delta= A
+ B$ with $A = L^\Delta_{14}$. By $A_\theta$ and $B_\theta$ we
denote the unbounded operators $(A,\mathcal{D}_\theta)$ and
$(B,\mathcal{D}_\theta)$, respectively. Likewise, we introduce
$A_{\theta'\theta}$ and $B_{\theta'\theta}$, $\theta'>\theta$. Their
operator norms are to be estimated by means of
(\ref{O2}) -- (\ref{O4}) and the following inequalities
\begin{equation*}
    x e^{-ax} \leq \frac{1}{ae}, \qquad     x^2 e^{-ax} \leq \frac{4}{(ae)^2}, \qquad a >
    0.
\end{equation*}
After some calculations we then get
\begin{equation}
  \label{O10}
  \|A_{\theta'\theta}\| \leq \frac{2c_1^{\rm max}}{e^2
  (\theta'-\theta)^2}, \qquad  \|B_{\theta'\theta}\| \leq \frac{\beta(\theta)}{e
  (\theta'-\theta)},
\end{equation}
with $\beta(\theta)$ given in (\ref{o26}). Now, for  $\theta' >
\theta$  and $t>0$, we define a bounded linear (multiplication)
operator $S_{\theta' \theta}(t):\mathcal{K}_{\theta}\to
\mathcal{K}_{\theta'}$ by the formula
\begin{equation}
  \label{o21}
(S_{\theta' \theta}(t) k)(\eta) = e^{- \Psi(\eta) t} k(\eta),
\end{equation}
and by $S_{\theta' \theta}(0)$ we will mean the corresponding
embedding operator. Then, for each $k\in \mathcal{K}_{\theta}$, the
map $[0,+\infty) \ni t \mapsto S_{\theta' \theta}(t) k \in
\mathcal{K}_{\theta}$ is continuous since
\begin{equation}
  \label{o22}
    \|S_{\theta' \theta} (t) k - S_{\theta' \theta} (t') k \|_{\theta'}
    \leq |t-t'| \cdot \frac{2 \coMax \|k\|_\theta}{(\theta' - \theta)^2
    e^2},
\end{equation}
that readily follows by (\ref{O10}). Note that the multiplication
operator by $\exp(-t \Psi)$ acts from $\mathcal{K}_{\theta}$ to
$\mathcal{K}_{\theta}$ for any $\theta$; hence,
$S_{\theta'\theta}(t): \mathcal{K}_\theta \to
\mathcal{D}_{\theta'}$, see (\ref{o11}). We define it, however, as
above in order to have the continuity secured by the estimate in
(\ref{o22}). By (\ref{o21}), for any $\theta'' \in (\theta,
\theta')$, we have that
\begin{equation}\label{SPrime}
    \frac{d}{d t}S_{\theta' \theta}(t) = A_{\theta'\theta''} S_{\theta''
    \theta}(t) =  A_{\theta'} S_{\theta'
    \theta}(t).
\end{equation}
Also by (\ref{o21}) it follows that
\begin{equation}\label{SNorm}
    \|S_{\theta' \theta} (t) k \|_{\theta'} \leq  \| k \|_{\theta}.
\end{equation}
Let $O$ be an operator acting in each $\mathcal{K}_\theta$ such
that: (a) $O:\mathcal{D}_\theta \to \mathcal{K}_\theta$; (b)
$O:\mathcal{K}_\theta \to \mathcal{K}_{\theta'}$ is a bounded
operator whenever $\theta'>\theta$. As in the case of $A$ and $B$,
we define the operators $O_\theta = (O, \mathcal{D}_\theta)$ and
$O_{\theta'\theta}$. Similarly as in (\ref{O5}), for these
operators, we have
\begin{equation}
  \label{S2}
  O_{\theta'}k = O_{\theta'\theta}k =
  S_{\theta'\theta''}(0)O_{\theta''\theta}k, \qquad k\in
  \mathcal{K}_\theta,
\end{equation}
where the second equality holds for all $\theta'' \in (\theta,
\theta')$.

Now we can turn to constructing the resolving operators
$Q_{\alpha_*\alpha_0}(t)$, see (\ref{SiQ}). For  a given $n\in
\mathds{N}$ and  $q > 1$, we introduce
\begin{gather}
\label{oo}
    \alpha_{2k+1} = \alpha_0 + \left( \frac{k+1}{n+1} \cdot \frac{q-1}{q} + \frac{k}{n} \cdot \frac{1}{q} \right) (\alpha_* -
    \alpha_0),
    \\[.2cm] \nonumber
    \alpha_{2k} = \alpha_0 + \left( \frac{k}{n+1} \cdot \frac{q-1}{q} + \frac{k}{n} \cdot \frac{1}{q} \right) (\alpha_* - \alpha_0), \qquad 0 \leq k \leq n.
\end{gather}
In particular $\alpha_{2n+1} = \alpha_*$.  For these $\alpha_l$,
$l=0,\dots, 2n+1$ and  $0 \leq t_n \leq t_{n-1} \leq
\ldots \leq t_2 \leq t_1 \leq t$, we introduce the operator
$\pi_{\alpha_* \alpha_0}^{(n)}(t, t_1, \dots, t_n):
\mathcal{K}_{\alpha_0} \rightarrow \mathcal{K}_{\alpha_*}$ as
follows
\begin{eqnarray}
  \label{o23}
\pi_{\alpha_* \alpha_0}^{(0)}(t) & = & S_{\alpha_*
\alpha_0}^{(0)}(t), \\[.2cm] \nonumber
  \pi_{\alpha_* \alpha_0}^{(n)}(t, t_1, \dots, t_n) &= & S_{\alpha_{2n+1} \alpha_{2n}} (t-t_1) B_{\alpha_{2n} \alpha_{2n-1}}
  S_{\alpha_{2n-1} \alpha_{2n-2}}
  (t_1-t_2) \\[.2cm] \nonumber
& \times &    B_{\alpha_{2n-2} \alpha_{2n-3}} \cdots S_{\alpha_{3}
\alpha_{2}}(t_{n-1} - t_n) B_{\alpha_{2} \alpha_{1}} S_{\alpha_{1}
     \alpha_{0}}(t_n), \quad n\in \mathds{N}.
\end{eqnarray}
Similarly as in obtaining the second equality in (\ref{S2}), we
conclude that $\pi_{\alpha_* \alpha_0}^{(n)}(t, t_1, ..., t_n)$ is
independent of the particular choice of the partition of
$(\alpha_0,\alpha_*)$ into subintervals $(\alpha_l,\alpha_{l+1})$.
In view of \eqref{SPrime}, we have that
\begin{equation}
  \label{o24}
    \frac{d}{dt} \pi_{\alpha_* \alpha_0}^{(n)}(t, t_1, \dots, t_n) = A_{\alpha_* \alpha} \pi_{\alpha \alpha_0}^{(n)}(t, t_1, \dots,
    t_n) = A_{\alpha_*} \pi_{\alpha_* \alpha_0}^{(n)}(t, t_1, \dots,
    t_n),
\end{equation}
holding for all $\alpha \in (\alpha_0, \alpha_*)$. For the same
$\alpha$, by setting in (\ref{o23}) $t_1 = t$ we obtain
\begin{eqnarray}
  \label{o24a}
  \pi^{(n)}_{\alpha_*\alpha_0} (t, t, t_2, \dots, t_n) & = &
  S_{\alpha_*\alpha_{2n}}(0) B_{\alpha_{2n}\alpha_{2n-1}}\pi^{(n-1)}_{\alpha_{2n-1}\alpha_0} ( t, t_2, \dots,
  t_n) \\[.2cm] & = & B_{\alpha_{*}\alpha}\pi^{(n-1)}_{\alpha\alpha_0} ( t, t_2, \dots,
  t_n), \nonumber
\end{eqnarray}
see (\ref{S2}). By \eqref{SNorm} and the second estimate in
(\ref{O10}) we get the following estimate of the operator norm of
(\ref{o23})
\begin{equation}
  \label{o25}
\| \pi_{\alpha_* \alpha_0}^{(n)}(t, t_1, ..., t_n) \| \leq
\prod\limits_{k = 1}^n \left(  \frac{q n
\beta(\alpha_{2k-1})}{(\alpha_* - \alpha_0)e} \right) \leq \left(
\frac{n}{e} \right)^n \cdot \left(\frac{q}{T(\alpha_*, \alpha_0)}
\right)^n .
\end{equation}
Now we set  $Q_{\alpha_* \alpha_0}^{(0)}(t) = S_{\alpha_*
\alpha_0}(t)$ and
\begin{equation}
\label{o25a} Q_{\alpha_* \alpha_0}^{(n)}(t) =     \suml{k=0}^n
\int\limits_0^t \int\limits_0^{t_1} \ldots \int\limits_0^{t_{k-1}}
\pi_{\alpha_* \alpha_0}^{(k)}(t, t_1, ..., t_k) dt_1 dt_2 \ldots
dt_k, \quad n \in \mathds{N}.
\end{equation}
Then by (\ref{o25}) it follows that
\begin{gather*}
    \|Q_{\alpha_* \alpha_0}^{(n)}(t) - Q_{\alpha_* \alpha_0}^{(n-1)}(t)\| \leq \int\limits_0^t \int\limits_0^{t_1} \ldots \int\limits_0^{t_{n-1}} \left(\frac{n}{e} \right)^n \cdot \left(\frac{q}{T(\alpha_*, \alpha_0)} \right)^n dt_1 dt_2 \ldots dt_n =
    \\[.2cm]
    = \frac{1}{n!} \left(\frac{n}{e} \right)^n \cdot \left(\frac{qt}{T(\alpha_*, \alpha_0)}\right)^n .
\end{gather*}
For each $\tau < T(\alpha_*, \alpha_0)$, by using (\ref{o26}) we
conclude that there exist $q
> 1$ and $\alpha\in (\alpha_0,\alpha_*)$ such that $q\tau <
T(\alpha, \alpha_0)$.  Then by the above estimate it follows that,
uniformly on $[0,\tau]$, $\{Q_{\alpha\alpha_0}^{(n)}(t)\}$ is a
Cauchy sequence with respect to the operator norm. Let $Q_{\alpha
\alpha_0}(t)$ be its limit. Clearly, this also applies to the
sequence $\{Q_{\alpha_*\alpha_0}^{(n)}(t)\}$, which therefore
converges to $Q_{\alpha_* \alpha_0}(t)$ in the same sense. By this
we have that: (a) for each $t\in [0,T(\alpha_*,\alpha_0))$, there
exists $\alpha\in (\alpha_0,\alpha_*)$ such that
\begin{equation}
  \label{o27}
\forall k \in \mathcal{K}_{\alpha_0} \qquad Q_{\alpha_* \alpha_0}(t)
k = Q_{\alpha \alpha_0}(t) k;
\end{equation}
(b) for each $k \in \mathcal{K}_{\alpha_0}$, the map
$[0,T(\alpha_*,\alpha_0))\ni t \mapsto Q_{\alpha_* \alpha_0}(t) k
\in \mathcal{K}_{\alpha_*}$ is continuous, and
\begin{equation*}
\forall k \in \mathcal{K}_{\alpha_0} \qquad Q_{\alpha_* \alpha_0}(t)
k \in \mathcal{D}_{\alpha_*}.
\end{equation*}
The latter follows by (\ref{o11}) and (\ref{o27}). In the sequel, we
will use the following estimate
\begin{equation}
  \label{o28a}
  \|Q_{\alpha_* \alpha_0}(t) \| \leq
  \frac{T(\alpha_*,\alpha_0)}{T(\alpha_*,\alpha_0)-t},
\end{equation}
that readily follows by (\ref{o25}).

For $n\in \mathds{N}$ and $\alpha\in (\alpha_0,\alpha_*)$, by
(\ref{o24}) and (\ref{o24a}) we obtain from (\ref{o25a}) that
\begin{eqnarray}
  \label{o29}
 \frac{d}{dt} Q^{(n)}_{\alpha_*\alpha_0} (t) = A_{\alpha_*\alpha}
 Q^{(n)}_{\alpha\alpha_0}(t) + B_{\alpha_*\alpha}
 Q^{(n-1)}_{\alpha\alpha_0} (t).
\end{eqnarray}
Fix $\tau< T(\alpha_*, \alpha_0)$ and then pick $\alpha \in
(\alpha_0, \alpha_*)$ such that $q\tau < T(\alpha, \alpha_0)$. By
the arguments used above the right-hand side of (\ref{o29})
converges as $n\to +\infty$, uniformly on $[0,\tau]$, to
\[
A_{\alpha_*\alpha}
 Q_{\alpha\alpha_0}(t) + B_{\alpha_*\alpha}
 Q_{\alpha\alpha_0} (t) = A_{\alpha_*}
 Q_{\alpha_*\alpha_0}(t) + B_{\alpha_*}
 Q_{\alpha_*\alpha_0} (t) = {\rm RHS}(\ref{S1}).
\]
This completes the proof that $k_t$ given in (\ref{SiQ}) is a
solution of the problem in (\ref{OCFE}) in the sense of Definition
\ref{Odf}. The uniqueness stated in the theorem can be obtained
similarly as in the proof of the same property in \cite[Lemma
4.1]{Kawasaki2}.

\section{Proof of Theorem \ref{Th2}}\label{Section:ProofTh2}

In this case, the proof is much longer and will be done in several
steps. In view of (\ref{O1}), the solution described by
Theorem \ref{Th1} has the property $k_t(\emptyset)=k_0(\emptyset)$
for all $t< T(\alpha_*, \alpha_0)$ since $(L^\Delta k)(\emptyset)
=0$. By the very choice of the spaces (\ref{KSpace}) this solution
satisfies condition (b) of Proposition \ref{CorFunProp}. Thus, it
remains to prove that it has the positivity property defined in
(\ref{o7}). To this end we make the following. First, in subsection
\ref{Subsection:AuxiliaryModel} we introduce an auxiliary model,
described by $L^\sigma$ with some $\sigma>0$. For this model, by
repeating the proof of Theorem \ref{Th1} we obtain the evolution
$k_0 \to k_t^\sigma$ in $\mathcal{K}_\theta$-spaces. In subsection
\ref{Subsection:SigmaConv}, we prove that
\begin{equation}
  \label{S10}
  \langle \! \langle G, k^\sigma_t \rangle \!\rangle \to \langle \! \langle G, k_t \rangle
  \!\rangle , \qquad \sigma \to 0^+,
\end{equation}
holding for all $G\in \mathcal{B}_{bs}$, cf. (\ref{o7}) and
(\ref{o8}). In the proof, we use the predual evolution constructed in
subsection \ref{Subsection:PredualEvolution}. To show that
$k_t^\sigma$ has the positivity property (\ref{o7}) we construct its
approximations (subsection \ref{Subsection:LocalEvolution}). As we
then show, these approximations coincide with the directly obtained
local correlation functions, see \eqref{LocalCorFun} and Corollary
\ref{Aco}, that have the required positivity by construction.
Finally, in subsection \ref{Subsection:LambdaNConv} we eliminate the
approximation and thus obtain the desired positivity of
$k^\sigma_t$.

\subsection{Auxiliary model}\label{Subsection:AuxiliaryModel}
For a given $\sigma
> 0$, we set $\psi_\sigma(x) = e^{- \sigma |x|^2}$, $x \in \Rd$.
Obviously
\[
    \intRd e^{- \sigma |x|^2} dx = \left(\frac{\pi}{\sigma} \right)^{d/2}.
\]
The model in question is described by
\begin{align}\label{OperatorLSigma}
(L^\sigma F)(\gamma) &= \suml{\{x,y\} \subset \gamma} \
\intRd \psi_\sigma(z) c_1(x,y;z) \Big(F\big(\gamma \bs \{x,y\} \cup z \big) - F(\gamma) \Big) dz \nonumber \\[.2cm]
&+ \suml{x \in \gamma} \  \intRd \psi_\sigma(y)
\tilde{c}_2(x;y;\gamma) \Big(F\big(\gamma \bs x \cup y\big) -
F(\gamma) \Big) dy.
\end{align}
Then we repeat the steps made in (\ref{o2}) and (\ref{o9}) to obtain
the operator $L^{\Delta,\sigma}=L^{\Delta, \sigma}_1 + L^{\Delta,
\sigma}_2$ in the following form, cf. (\ref{O1}),
\begin{align*}
(L^{\Delta, \sigma}_1 k)(\eta) &= \ \ \frac{1}{2}
\int\limits_{(\mathds{R}^d)^2} \sum\limits_{z \in \eta}
\psi_\sigma(z) c_1(x,y;z) k(\eta \backslash z \cup \{x,y\})  dx dy
    \\
&- \frac{1}{2} \int\limits_{(\mathds{R}^d)^2} \psi_\sigma(z)
\sum\limits_{x \in \eta} c_1(x,y;z) k(\eta \cup y)  dy dz
    \\
    &- \frac{1}{2} \int\limits_{(\mathds{R}^d)^2} \psi_\sigma(z) \sum\limits_{y \in \eta} c_1(x,y;z) k(\eta \cup x)  dx dz
    \\
    &- \int\limits_{\mathds{R}^d} \psi_\sigma(z) \sum\limits_{\{x,y\} \subset \eta}  c_1(x,y;z) dz \ k(\eta)
    \\
(L^{\Delta, \sigma}_{2} k)(\eta) &= \intRd \sum\limits_{y \in \eta}
(Q_y k)(\eta \backslash y \cup x) \psi_\sigma(y) c_2(x-y) \prodl{u
\in \eta \backslash y} e^{-\phi(y-u)} dx
    \\
    &- \intRd (Q_y k)(\eta) \psi_\sigma(y) \sum\limits_{x \in \eta} c_2(x-y) \prodl{u \in \eta \backslash x} e^{-\phi(y-u)}
    dy,
\end{align*}
where $Q_y$ is the same as in \eqref{QyNotion}. Like above, cf.
(\ref{O1}), we split $L^{\Delta,\sigma}_1$ (resp.
$L^{\Delta,\sigma}_2$) into the summands $L_{11}^{\Delta, \sigma}$,
$L_{12}^{\Delta, \sigma}$, $L_{13}^{\Delta, \sigma}$,
$L_{14}^{\Delta, \sigma}$ (resp. $L_{21}^{\Delta, \sigma}$ and
$L_{22}^{\Delta, \sigma}$). We also introduce
\begin{gather}
  \label{S11}
  (A^\sigma k)(\eta) = - \Psi_\sigma (\eta) k(\eta), \qquad   B^\sigma = L^{\Delta,\sigma} - A^\sigma, \\[.2cm] \nonumber
\Psi_\sigma
  (\eta):= \int\limits_{\mathds{R}^d} \psi_\sigma(z) \sum\limits_{\{x,y\} \subset \eta}  c_1(x,y;z)
  dz,
\end{gather}
and then define the operators $A^\sigma_\theta=(A^\sigma,
\mathcal{D}_\theta)$, $B^\sigma_\theta=(B^\sigma,
\mathcal{D}_\theta)$, $L^{\Delta,\sigma}_\theta =
A^\sigma_\theta+B^\sigma_\theta$, $A^\sigma_{\theta'\theta}$,
$B^\sigma_{\theta'\theta}$ for $\theta'>\theta$ and $\mathcal{D}_\theta$ defined in (\ref{o10}). Since
$\psi_{\sigma} \leq 1$, by the literal repetition of the proof of
Theorem \ref{Th1} we construct the family of operators
$Q_{\alpha_*\alpha_0}^\sigma (t)$, $\alpha_0\in \mathds{R}$,
$\alpha_*>\alpha$, $t\in [0, T(\alpha_*,\alpha_0))$ such that
$k_t^\sigma = Q_{\alpha_*\alpha_0}^\sigma (t)k_0$ with $k_0 \in
\mathcal{K}_{\alpha_0}$is the unique classical solution -- on the
time interval $[0,T(\alpha_*,\alpha_0))$ with $T(\alpha_*,\alpha_0)$
as in (\ref{o26}) -- of the problem
\begin{equation}\label{SigmaEv}
\frac{d}{dt} k_t^\sigma = L^{\Delta, \sigma}_{\alpha_*} k_t^\sigma,
\qquad k^\sigma_t|_{t=0} = k_0 .
\end{equation}
Note that the norm of $Q_{\alpha_*\alpha_0}^\sigma (t)$ also
satisfies (\ref{o28a}).

\subsection{Predual evolution}\label{Subsection:PredualEvolution}
To prove (\ref{S10}) we allow $G$ to evolve accordingly to the rule
\begin{equation*}
  \langle \! \langle G_t, k_0 \rangle \! \rangle = \langle \! \langle G_0, Q_{\alpha_*\alpha_0}(t) k_0 \rangle \! \rangle
\end{equation*}
The proper context to this is to construct the corresponding evolution
in the space predual to $\mathcal{K}_{\alpha_*}$, which
ought to be of $L^1$-type. For $\theta \in \mathds{R}$, we introduce
\begin{equation}\label{GThetaSpace}
 \mathcal{G}_\theta =\{G:\Gamma_0\to \mathds{R}: |G|_\theta < \infty\}, \qquad    |G|_\theta := \intGamN |G(\eta)| e^{\theta |\eta|} \lambda(d\eta).
\end{equation}
Obviously, for $\theta' > \theta$, we have that
$\mathcal{G}_{\theta'} \hookrightarrow \mathcal{G}_\theta$. Notably,
$G \in \bbs$ lies in $\mathcal{G}_\theta$ with an arbitrary $\theta
\in \mathds{R}$. Indeed, let $M$ be the bound of $|G|$ and $N$ and
$\Lambda$ be as in Definition \ref{Gdf}. Then we have
\begin{equation}\label{BbsAndG}
\intGamN |G(\eta)| e^{\theta |\eta|} \lambda (d \eta)= \suml{n=0}^N
\frac{1}{n!} e^{\theta n} \intl{\Lambda^n} |G(x_1, \ldots, x_n)|
dx_1 \ldots dx_n \leq M e^{|\Lambda| e^\theta}.
\end{equation}
\begin{lemma}
  \label{Slm}
Let $Q_{\alpha_*\alpha_0}(t)$, $\alpha_0\in \mathds{R}$,
$\alpha_*>\alpha_0$,  $t<T(\alpha_*,\alpha_0)$, see (\ref{o26}), be
the family of bounded operators constructed in the proof of Theorem
\ref{Th1}. Then there exists the family $H_{\alpha_0\alpha_*} (t):
\mathcal{G}_{\alpha_*} \to \mathcal{G}_{\alpha_0}$,
$t<T(\alpha_*,\alpha_0)$ such that: (a) the norm of
$H_{\alpha_0\alpha_*}(t)$ satisfies (\ref{o28a}); (b) for
each $G\in \mathcal{G}_{\alpha_*}$ and $k\in
\mathcal{K}_{\alpha_0}$, the following holds
\begin{equation}
  \label{QHDUality}
  \langle \! \langle H_{\alpha_0\alpha_*} (t) G, k \rangle \! \rangle = \langle \! \langle G, Q_{\alpha_*\alpha_0}(t) k \rangle \!
  \rangle;
\end{equation}
(c) the map $[0,T(\alpha_*,\alpha_0)) \ni t \mapsto
H_{\alpha_0\alpha_*} (t)$ is continuous  in the operator norm
topology.
\end{lemma}
\begin{proof}
Clearly, the most challenging part is the continuity stated in (c).
Thus, we start by deriving the corresponding generating operator. To
this end, we use the rule
\[
 \langle \! \langle G, L^\Delta k \rangle \! \rangle = \langle \! \langle\hat{L} G,  k \rangle \!
  \rangle .
\]
It can be shown, see \cite{RepulsiveCoalescence},
that it has the following form
\begin{equation*}
    \hat{L} = \hat{L}_1 + \hat{L}_2,
\end{equation*}
with
\begin{eqnarray*}
  & &  \qquad \quad ( \hat{L}_1 G) (\eta) = \int\limits_{\mathds{R}^d} \sum\limits_{\{x,y\} \subset \eta}
   c_1(x,y;z) \\[.2cm] \nonumber & & \times
    \Big(G(\eta \backslash \{x,y\} \cup z) - G(\eta \backslash y) -G(\eta \backslash x) -G(\eta) \Big) dz,
    \\[.2cm] \nonumber
    & & (\hat{L}_2 G )(\eta) = \intRd \suml{x \in \eta} c_2(x-y) \suml{\xi \subset \eta \bs x} \big[G(\xi \cup y) - G(\xi \cup x)\big] \\[.2cm]
     \nonumber & & \qquad  \times \suml{\zeta \subset \xi} \prodl{u \in \eta \bs \xi \bs x \cup \zeta} \left(e^{-\phi (y -u)} - 1\right)
     dy.
\end{eqnarray*}
Now we set
\begin{equation*}
\hat{L} = \hat{A} + \hat{B}, \qquad  ( \hat{A} G) (\eta)
:= - \Psi(\eta) G(\eta),
\end{equation*}
where $\Psi$ is as in the last line of (\ref{O1}). Then define,
cf. (\ref{o10}),
\begin{equation}
  \label{hD}
\hat{ \mathcal{D}}_\theta = \{G\in \mathcal{G}_\theta: |\cdot|^2 G
\in \mathcal{G}_\theta\}, \qquad \theta \in \mathds{R}.
\end{equation}
Like in the dual spaces $\mathcal{K}_\theta$, cf. (\ref{o11}), here
we have that both $\hat{A}$ and $\hat{B}$ map $\hat{
\mathcal{D}}_\theta$ in $\mathcal{G}_\theta$. This allows one to
introduce the operators $\hat{A}_\theta
=(\hat{A},\hat{\mathcal{D}}_\theta)$ and $\hat{B}_\theta
=(\hat{B},\hat{\mathcal{D}}_\theta)$ as well as bounded operators
$\hat{A}_{\theta \theta'}$ and $\hat{B}_{\theta \theta'}$ mapping
$\mathcal{G}_{\theta'}$ to $\mathcal{G}_\theta$ for
$\theta'>\theta$. Their operator norms satisfy the same estimates as
the norms of $A_{\theta'\theta}$ and $B_{\theta'\theta}$,
respectively, see (\ref{O10}).

For such $\theta$ and $\theta'$, we also set
$\hat{S}_{\theta \theta'}(t): \mathcal{G}_{\theta'}\to
\mathcal{G}_\theta$  to be the multiplication operator by the
function $\exp(- t \Psi(\eta))$. Similarly as in (\ref{o22}) one
shows that
\begin{equation}
  \label{S15}
\left\vert S_{\theta \theta'} (t) G - S_{\theta \theta'} (t')
G\right\vert_\theta \leq |t-t'| \frac{2 c_1^\text{max}}{(\theta' -
\theta)^2 e^2} |G|_{\theta'},
\end{equation}
which yields the continuity of the map $[0,+\infty)\ni t \mapsto
\hat{S}_{\theta \theta'} (t)$ in the operator norm topology. By the
very construction of these operators we have that, for each $G\in
\mathcal{G}_{\theta'}$ and $k \in \mathcal{K}_\theta$, the following
holds
\begin{equation}
 \label{S16}
 \langle \!\langle \hat{B}_{\theta \theta'} G, k \rangle \!\rangle = \langle \!\langle G, B_{\theta' \theta} k \rangle \!\rangle , \qquad
 \langle \!\langle \hat{S}_{\theta \theta'} (t) G, k \rangle \!\rangle = \langle \!\langle G, S_{\theta' \theta} (t) k \rangle \!\rangle,
\end{equation}
where the second equality holds for all $t\geq0$. Now, for a given
$n\in \mathds{N}$, $\alpha_l$, $l=0,\dots , 2n+1$ defined in
(\ref{oo}) and $t_1, \dots , t_n$ as in (\ref{o23}), we set
\begin{gather*}
    \varpi^{(n)}_{\alpha_0 \alpha_*} (t, t_1, \ldots, t_n) = \hat{S}_{\alpha_0 \alpha_1} (t_n) \hat{B}_{\alpha_1 \alpha_2}
    \hat{S}_{\alpha_2 \alpha_3} (t_{n-1} - t_n) \hat{B}_{\alpha_3
    \alpha_4} \times
    \cdots \times
    \\[.2cm] \nonumber
   \times  \hat{S}_{\alpha_{2n-2} \alpha_{2n-1}} (t_{1} - t_2) \hat{B}_{\alpha_{2n-1} \alpha_{2n}} \hat{S}_{\alpha_{2n} \alpha_{2n+1}} (t - t_1).
\end{gather*}
Then we define
\begin{equation}\label{OperatorHn}
    H_{\alpha_0 \alpha_*}(t) = S_{\alpha_* \alpha_0}(t) + \suml{n=1}^{\infty} \int\limits_0^t \int\limits_0^{t_1} \ldots
    \int\limits_0^{t_{n-1}} \varpi_{\alpha_0 \alpha_*}^{(n)}(t, t_1, ..., t_n) dt_1 dt_2 \ldots dt_n.
\end{equation}
Since the operator norms of all $\hat{S}$ and $\hat{B}$ satisfy the
same estimates as the norms of respectively $S$ and $B$, the
operator norm of $\varpi^{(n)}_{\alpha_0\alpha_**}$ satisfies
(\ref{o25}). Hence, the series in (\ref{OperatorHn}) converges in
the norm topology, uniformly on compact subsets of
$[0,T(\alpha_*,\alpha_0))$, which together with (\ref{S15}) yields
the continuity stated in claim (c) and the bound stated in (a). In
view of the convergence just mentioned, to prove (\ref{QHDUality})
it is enough to show that, for each $n\in \mathds{N}$ and $0\leq t_n
\leq t_{n-1} \leq \cdots \leq t_1 \leq t$, the following holds
\begin{gather*}
\langle\! \langle \varpi^{(n)}_{\alpha_0 \alpha_*} (t, t_1, \ldots,
t_n) G, k \rangle \! \rangle = \langle\! \langle
 G, \pi^{(n)}_{\alpha_* \alpha_0} (t, t_1, \ldots, t_n) k \rangle
\! \rangle,
\end{gather*}
which is obviously the case in view of (\ref{S16}).
\end{proof}

\subsection{Taking the limit $\sigma \rightarrow 0$}\label{Subsection:SigmaConv}

Our aim now is to prove the following statement, cf. (\ref{S10}).
\begin{lemma}
  \label{S1lm}
For arbitrary $\alpha_0\in \mathds{R}$, $\alpha_*>\alpha_0$, every
$G\in \mathcal{B}_{bs}$ and $k_0\in \mathcal{K}_{\alpha_0}$, the
following holds
\begin{equation}
  \label{S20}
 \forall t<T(\alpha_*, \alpha_0)/2 \qquad  \langle \! \langle G, Q^\sigma_{\alpha_* \alpha_0} (t) k_0\rangle \!
  \rangle \to \langle \! \langle G, Q_{\alpha_* \alpha_0} (t) k_0\rangle \!
  \rangle, \qquad \sigma\to 0^+.
\end{equation}
\end{lemma}
\begin{proof}
First of all we note that, for each $\alpha_2>\alpha_1$, both
$Q_{\alpha_2\alpha_1}(0)$ and $Q^\sigma_{\alpha_2\alpha_1}(0)$ are
the corresponding embedding operators. Then, for $\alpha_0 <
\alpha_1 < \alpha_2 < \alpha_*$, we can write, cf. (\ref{S1}) and
(\ref{O5}),
\begin{gather}\label{SigmaLim}
[Q_{\alpha_*\alpha_0}(t) - Q^\sigma_{\alpha_*\alpha_0}(t)] k_0 = -
\int_{0}^{t} \frac{d}{ds}
[Q_{\alpha_*\alpha_2}(t-s)Q^\sigma_{\alpha_2\alpha_0}(s)] k_0 ds \\[.2cm] \nonumber =
\int_{0}^{t} Q_{\alpha_* \alpha_2} (t-s) (A_{\alpha_2\alpha_1} -
A^\sigma_{\alpha_2\alpha_1}) k^\sigma_s ds \\[.2cm] \nonumber + \int_{0}^{t}
Q_{\alpha_* \alpha_2}(t-s) (B_{\alpha_2\alpha_1} -
B^\sigma_{\alpha_2\alpha_1}) k^\sigma_s ds,
\end{gather}
where $k_s^\sigma$ is supposed to lie in $\mathcal{K}_{\alpha_1}$
and
\begin{equation}
  \label{S21}
t < \min\{T(\alpha_1, \alpha_0); T(\alpha_*, \alpha_2)\}.
\end{equation}
Then
\begin{equation}
  \label{S21a}
\langle \! \langle G, Q_{\alpha_* \alpha_0} (t) k_0\rangle \!
  \rangle - \langle \! \langle G, Q^\sigma_{\alpha_* \alpha_0} (t) k_0\rangle \!
  \rangle = \Upsilon^1_\sigma(t) + \Upsilon^2_\sigma(t),
\end{equation}
where $\Upsilon^1_\sigma(t)$ and $\Upsilon^2_\sigma(t)$ correspond
to the first and second summands in the right-hand side of
(\ref{SigmaLim}), respectively.  By means of (\ref{QHDUality}) we
obtain
\begin{gather}
\label{S22} \Upsilon^1_\sigma(t)= \intGamN G(\eta)\left(
\int_{0}^{t} Q_{\alpha_* \alpha_2}(t-s) (A_{\alpha_2
\alpha_1}-A_{\alpha_2 \alpha_1}^\sigma) k^\sigma_s(\eta) ds
\right)\lambda(d\eta)
    \\[.2cm] \nonumber
= -\int_{0}^{t}\left( \intGamN G_{t-s}(\eta) \left[
\Psi(\eta)-\Psi_\sigma (\eta)\right] k_s^\sigma(\eta)
\lambda(d\eta)\right) ds,
\end{gather}
where $t$ (hence $t-s$ and $s$) satisfy (\ref{S21}), and $G_{t-s}:=
H_{\alpha_2 \alpha_*}(t-s) G \in \mathcal{G}_{\alpha_2}$ since $G\in
\mathcal{G}_{\alpha_*}$, see (\ref{BbsAndG}). By
(\ref{KNormEstimate}) and then by (\ref{o28a}) we get from
(\ref{S22}) the following estimate
\begin{gather}
\label{S22a} \left|\Upsilon^1_\sigma(t)\right| \leq
\frac{\|k_0\|_{\alpha_0} T(\alpha_1, \alpha_0)}{T(\alpha_1,
\alpha_0) - t} \int_0^t \int_{\Gamma_0} \left| G_s (\eta)\right|
\left[\Psi(\eta) - \Psi_\sigma(\eta) \right]e^{\alpha_1|\eta|} d s
\lambda (d\eta)
\end{gather}
For each $\eta\in \Gamma_0$, the integral in the right-hand side of
the second line in (\ref{S11}) is bounded by $c_1^{\rm
max}|\eta|(|\eta|-1)/2$, see the second line in
(\ref{ModelAssumptions}). Then by Lebesgue's dominated convergence
theorem we conclude
\begin{equation}
\label{S22b} \forall \eta \in \Gamma_0 \qquad \Psi_\sigma (\eta) \to
\Psi(\eta).
\end{equation}
At the same time, the integral over $[0,t]\times \Gamma_0$ in
(\ref{S22a}) is bounded by
\begin{gather*}
 c_1^{\rm max} \int_0^t \left(\int_{\Gamma_0} |\eta|^2e^{-(\alpha_2 -
\alpha _1)|\eta|} \left| G_s (\eta)\right|e^{\alpha_2|\eta|} \lambda
(d\eta)\right) d s \\[.2cm] \leq \frac{4c_1^{\rm max}
}{\left(e(\alpha_2-\alpha_1) \right)^2} \int_0^t |G_s|_{\alpha_2} ds
\leq \frac{4c_1^{\rm max} t T(\alpha_*,
\alpha_2)}{\left(e(\alpha_2-\alpha_1) \right)^2\left(T(\alpha_*,
\alpha_2)-t \right)},
\end{gather*}
where the latter estimate is obtained by claim (a) of Lemma
\ref{Slm}. This allows one to apply the Lebesgue dominated
convergence theorem to the mentioned integral in (\ref{S22a}), which
by (\ref{S22b}) yields
\begin{equation*}
\Upsilon^1_\sigma(t) \to 0, \qquad \sigma\to 0^+,
\end{equation*}
whenever $t$ satisfies (\ref{S21}).

The second summand in the right-hand side of (\ref{S21a}) is
\begin{eqnarray}
\label{S27} \Upsilon^2_\sigma(t) &= & \intGamN G(\eta)\left(
\int_{0}^{t} Q_{\alpha_* \alpha_2}(t-s) \left[B_{\alpha_2
\alpha_1}-B_{\alpha_2 \alpha_1}^\sigma\right] k^\sigma_s(\eta) ds
\right)\lambda(d\eta)
    \\[.2cm] \nonumber
& = & \int_{0}^{t}\left( \intGamN G_{t-s}(\eta) \left[B_{\alpha_2
\alpha_1}-B_{\alpha_2 \alpha_1}^\sigma \right] k_s^\sigma(\eta)
\lambda(d\eta)\right) ds\\[.2cm] \nonumber & = &
\Upsilon^{2,1}_\sigma(t) +\cdots + \Upsilon^{2,5}_\sigma(t),
\end{eqnarray}
For $i=1,2,3$, the summands $\Upsilon^{2,i}_\sigma(t)$ correspond to
$L^\Delta_{1i}$ with the same $i$; for $i=4,5$, they correspond to
$L^\Delta_{21}$ and $L^\Delta_{22}$, respectively, see (\ref{O1}).
To estimate $\Upsilon^{2,1}_\sigma(t)$, by \eqref{Minlos1},
\eqref{Minlos2} and (\ref{KNormEstimate}) we obtain
\begin{eqnarray}
\label{S27a}
 & & \left\vert\intGamN G_{t-s}(\eta) (L_{11}^\Delta -
L_{11}^{\Delta, \sigma})_{\alpha_2\alpha_1} k^\sigma_s(\eta)
\lambda(d\eta) \right\vert
    \\[.2cm] \nonumber
& &  \leq  \frac{1}{2} \intGamN |G_{t-s}(\eta)|
\int\limits_{\left(\Rd\right)^2} \suml{z\in \eta} (1 - \psi_{\sigma}
(z)) c_1(x,y;z) |k^\sigma_s(\eta \bs z \cup \{x,y\})| dx dy
\lambda(d\eta)
    \\[.2cm] \nonumber
& &   =  \intGamN \intRd\suml{\{x,y\} \subset \eta}  |G_{t-s}(\eta
\bs \{x,y\} \cup z)| (1 - \psi_{\sigma} (z)) c_1(x,y;z) dz
|k^\sigma_s(\eta)| \lambda(d\eta)   \\[.2cm] \nonumber
& &
     \leq  ||k^\sigma_s||_{\alpha_1} \intRd (1 - \psi_{\sigma}
(z)) \intGamN \suml{\{x,y\} \subset \eta}  |G_{t-s}(\eta \bs \{x,y\}
\cup z)|  c_1(x,y;z)   e^{\alpha_1 |\eta|} \lambda(d\eta) dz.
\end{eqnarray}
By this estimate we then get
\begin{eqnarray}
  \label{S28}
\left\vert \Upsilon^{2,1}_\sigma(t) \right\vert & = &  \left|
\int_{0}^{t} \intGamN G_{t-s}(\eta) (L_{11}^\Delta - L_{11}^{\Delta,
\sigma})_{\alpha_2\alpha_1} k^\sigma_s(\eta) \lambda(d\eta) ds  \right| \\[.2cm]
\nonumber & \leq & \intRd (1 - \psi_{\sigma} (z)) g(z) dz
\end{eqnarray}
where
\begin{gather*}
 g(z) = \intGamN \int_0^t \|k^\sigma_s\|_{\alpha_1} \left(\suml{\{x,y\} \subset \eta}  |G_{t-s}(\eta \bs \{x,y\} \cup z)|
 c_1(x,y;z)   e^{\alpha_1 |\eta|} \lambda(d\eta)\right) ds .
\end{gather*}
Let us show that $g$ is integrable whenever $t$ (hence $s$ and
$t-s$) satisfy (\ref{S21}). To this end by \eqref{Minlos1},
\eqref{Minlos2} and claim (a) of Lemma \ref{Slm} we obtain
\begin{eqnarray}
\label{S29} & &     \intRd g(z) dz \qquad \qquad  \\[.2cm] \nonumber & & \quad = \int_0^t
\|k^\sigma_s\|_{\alpha_1}\left( \intRd  \intGamN \suml{\{x,y\}
\subset \eta}  |G_{t-s}(\eta \bs \{x,y\} \cup z)|  c_1(x,y;z)
e^{\alpha_1 |\eta|} \lambda(d\eta) dz \right)ds
 \qquad  \\[.2cm] \nonumber & & \
    = \frac{e^{\alpha_1}}{2}  \int_0^t \|k^\sigma_s\|_{\alpha_1}
\intGamN e^{\alpha_2 |\eta|} |G_{t-s}(\eta)| e^{-(\alpha_2 -
\alpha_1) |\eta|} \left(\suml{z \in \eta} \int\limits_{(\Rd)^2}
c_1(x,y;z) dx dy \right) \lambda(d\eta) ds
    \\[.2cm] \nonumber & & \
    \leq  \frac{e^{\alpha_1} \cOne}{2(\alpha_2 - \alpha_1) e}
 \int_0^t \|k^\sigma_s\|_{\alpha_1} |G_{t-s}|_{\alpha_2}  ds
   \leq \frac{e^{\alpha_1} \cOne }{2(\alpha_2 - \alpha_1) e}
D_t(\alpha_2,\alpha_1),
\end{eqnarray}
with
\begin{equation}
  \label{S26}
D_t (\alpha_2, \alpha_1) = \frac{t T(\alpha_*, \alpha_2) T(\alpha_1,
\alpha_0)}{(T(\alpha_*, \alpha_2) - t)(T(\alpha_1, \alpha_0)-t)}.
\end{equation}
Then by (\ref{S28}) we obtain that
\begin{equation}
  \label{S26a}
\Upsilon^{2,1}_\sigma(t) \to 0, \qquad \sigma\to 0^+,
\end{equation}
whenever $t$ satisfies (\ref{S21}). By the literal repetition of the
arguments yielding (\ref{S26a}) we prove the same convergence to
zero also for $\Upsilon^{2,2}_\sigma(t)$ and
$\Upsilon^{2,3}_\sigma(t)$, cf. (\ref{O1}). To estimate
$\Upsilon^{2,4}_\sigma(t)$ similarly as in (\ref{S27a}), we write
\begin{eqnarray}
\label{S26b} \left|\Upsilon^{2,4}_\sigma(t)\right| & = & \left|
\int_{0}^{t} \intGamN G_{t-s}(\eta) (L_{21}^\Delta - L_{21}^{\Delta,
\sigma})_{\alpha_2\alpha_1} k^\sigma_s(\eta) \lambda(d\eta) ds
\right|
\\[.2cm] \nonumber
& \leq & \intRd  (1 - \psi_{\sigma}(y)) \int_{0}^{t}\bigg{(}
\intGamN \suml{x \in \eta} c_2(x;y) |G_{t-s}(\eta \bs x \cup y)|
\\[.2cm] \nonumber
& \times & \intGamN |k_s^\sigma(\eta \cup \xi)|  \prodl{u \in \xi}
\left( 1 - e^{-\phi(y-u)} \right) \lambda(d\xi)
\lambda(d\eta)\bigg{)} ds dy\\[.2cm] \nonumber
& = &  \intRd  (1 - \psi_{\sigma}(y)) h(y) dy,
\end{eqnarray}
with
\begin{eqnarray*}
h(y) & = & \int_{0}^{t} \intGamN \suml{x \in \eta} c_2(x-y)
|G_{t-s}(\eta \bs x \cup y)| \\[.2cm] \nonumber
& \times & \bigg{(} \intGamN |k_s^\sigma(\eta \cup \xi)|  \prodl{u
\in \xi} \left( 1 - e^{-\phi(y-u)} \right) \lambda(d\xi) \bigg{)}
\lambda(d\eta) ds.
\end{eqnarray*}
Analogously as in (\ref{S29}), we have
\begin{align*}
 \intRd h(y) dy &\leq \exp(\phinorm e^{\alpha_1}) \int_{0}^{t} \|k_s^\sigma\|_{\alpha_1} \intGamN e^{\alpha_1 |\eta|}
 |G_{t-s}(\eta)| \suml{y \in \eta} \ \intRd  c_2(x-y)  dx \lambda(d\eta) ds
    \\[.2cm]
    &\leq   \frac{\cTwo}{(\alpha_2 - \alpha_1) e} t D_t(\alpha_2,\alpha_1)\exp(\phinorm e^{\alpha_1}),
\end{align*}
where $D_t$ is the same as in (\ref{S26}). Then we apply  the same
dominated convergence theorem in the last line of (\ref{S26b}) and
obtain that
\begin{equation*}
\Upsilon^{2,4}_\sigma(t) \to 0, \qquad \sigma\to 0^+,
\end{equation*}
whenever $t$ satisfies (\ref{S21}). The proof of the same
convergence for $\Upsilon^{2,5}_\sigma(t)$ is completely analogous.
Thus, all the summands in the last line of (\ref{S27}) tend to zero
as $\sigma\to 0^{+}$ -- that yields (\ref{S20}) -- whenever $t$
satisfies (\ref{S21}). It remains to prove that, for each $t<
T(\alpha_*,\alpha_0)/2$, one can pick $\alpha_1, \alpha_2 \in
(\alpha_0, \alpha_*)$ such that (\ref{S21}) holds for these
$\alpha_2$ and $\alpha_1$. To this end, we fix
$t<T(\alpha_*,\alpha_0)/2$, take $\alpha_1 = (\alpha_* +
\alpha_0)/2$ and $\alpha_2 = \alpha_1 + \epsilon \beta(\alpha_*)$
with $\epsilon
> 0$ being chosen later and such that $\alpha_2 < \alpha_*$. For this choice, by
(\ref{o26}) we have that $T(\alpha_1, \alpha_0) \geq \frac{1}{2}
T(\alpha_*,\alpha_0)>t$ since $\beta(\theta)$ is increasing. At the
same time, $T(\alpha_*, \alpha_2)) +\epsilon =
\frac{1}{2}T(\alpha_*,\alpha_0)$. Then we take $\epsilon =
(\frac{1}{2} T(\alpha_*,\alpha_0) - t)/2$ in the choice of
$\alpha_2$, which yields $t< T(\alpha_*, \alpha_2)$.
\end{proof}

\subsection{Approximations}\label{Subsection:LocalEvolution}

Our aim now is to prove that $k^\sigma_t =
Q^\sigma_{\alpha_*\alpha_0} (t)k_0$ has the positivity property
defined in (\ref{o7}) whenever $k_0$ is the correlation
function of a certain $\mu_0 \in \mathcal{P}_{\rm exp}$. Then, for
$t<T(\alpha_*, \alpha_0)/2$, the same positivity property of $k_t=
Q_{\alpha_*\alpha_0} (t)k_0$ will follow by Lemma \ref{S1lm}.
Similarly as in \cite{Kawasaki2,SpatialEcologicalModel}, the main
idea of proving the positivity of $k^\sigma_t$ is to approximate it
by a correlation function of a finite system of this kind, which is
positive by Proposition \ref{CorFunProp}. Thereafter, one has to
prove that the positivity is preserved when the approximation is
eliminated.

\subsubsection{The approximate evolution} For a compact $\Lambda\subset \mathds{R}^d$, let
$\mu_0^{\Lambda}$ be the projection of the initial state $\mu_0$,
see (\ref{Projection}). Then its density $R^\Lambda_{\mu_0}$ and the
correlation function $k_0$ are related to each other in
(\ref{CorAndDens}). For $N\in \mathds{N}$, we set
\begin{equation}\label{DensExt}
    R_0^{\Lambda, N} (\eta) =
    \begin{cases}
        R_{\mu_0}^\Lambda(\eta) &\text{if } |\eta| \leq N \ \text{and} \ \eta \in \Gamma_\Lambda,
        \\[.2 cm]
        0 &\text{otherwise.}
    \end{cases}
\end{equation}
Note that $R_0^{\Lambda, N}:\Gamma_0 \to \mathds{R}$, unlike to
$R^\Lambda_{\mu_0}$ which is a function on $\Gamma_\Lambda$. Now we
define
\begin{equation}\label{kLambdaN}
    k_0^{\Lambda, N}(\eta) = \int\limits_{\Gamma_\Lambda} R_0^{\Lambda, N}(\eta \cup \xi) \lambda(d\xi), \qquad \eta \in \Gamma_0,
\end{equation}
By (\ref{DensExt}) and (\ref{kLambdaN}) we have that  $k_0^{\Lambda,
N} \leq k_{\mu_0}$, so that $k_{\mu_0} \in \mathcal{K}_{\alpha_0}$
implies $k_0^{\Lambda, N} \in \mathcal{K}_{\alpha_0}$. Then by
Theorem \ref{Th1},
\begin{equation}
  \label{A1}
k_t^{\Lambda, N} = Q^\sigma_{\alpha_0 \alpha_*}(t) k_0^{\Lambda, N}
\in \mathcal{K}_{\alpha_*}
\end{equation}
is the unique classical solution of the problem
\begin{equation*}
\frac{d}{dt} k_t^{\Lambda, N} = L^{\Delta, \sigma}_{\alpha_*}
k_t^{\Lambda, N}, \qquad k_t|_{t=0}^{\Lambda, N} = k_0^{\Lambda, N}
\in \mathcal{K}_{\alpha_0},
\end{equation*}
on the time interval $[0,T(\alpha_*,\alpha_0))$.
\begin{lemma}
  \label{Alm}
Let $k_t^{\Lambda,N}$, $t< T(\alpha_*,\alpha_0)$ be as in
(\ref{A1}). Then, for each $G\in \mathcal{B}_{bs}^*$ and $t<
T(\alpha_*,\alpha_0)$, the following holds
\begin{equation}
  \label{A2}
\langle \! \langle G,k_t^{\Lambda,N}\rangle \! \rangle \geq 0.
\end{equation}
\end{lemma}
The  proof of this statement will follow by Corollary \ref{Aco}
proved below in which we show that
\begin{equation}
  \label{A3}
k_t^{\Lambda,N} = q_t^{\Lambda,N},
\end{equation}
where
\begin{equation}\label{LocalCorFun}
q_t^{\Lambda, N} (\eta) = \intGamN R_t^{\Lambda, N} (\eta \cup \xi)
\lambda(d\xi),\qquad t \geq 0.
\end{equation}
Here $R_t^{\Lambda, N}$ is the (non-normalized) density obtained
from $R_0^{\Lambda, N}$ given in (\ref{DensExt}) in the course of
the evolution related to $L^\sigma$. By this fact, $q_t^{\Lambda,
N}$ satisfies (\ref{A2}), which will yield the proof. According to
this, we proceed by constructing the evolution $R_0^{\Lambda, N}\to
R_t^{\Lambda, N}$, which will allow us to use $q_t^{\Lambda, N}$
defined in (\ref{LocalCorFun}). The next step will be to prove
(\ref{A3}).

\subsubsection{The local evolution}
As just mentioned, the evolution $R_0^{\Lambda, N}\to R_t^{\Lambda,
N}$ is related to the local evolution of the auxiliary model
described by $L^\sigma$, see subsection
\ref{Subsection:AuxiliaryModel}. Here \emph{local} means the
following. Assume that the initial state $\nu_0$ is such that $\nu_0
(\Gamma_0)=1$. That is, the system is finite and hence local. Assume
also that it has density $R_{\nu_0} = \frac{d \nu_0}{d\lambda}$.
Then the evolution related to the  Kolmogorov equation with
$L^\sigma$ can be described as the evolution of densities by solving
the  corresponding Fokker-Planck equation
\begin{equation}
  \label{A4}
\frac{d}{dt} R_t = L^\dagger R_t, \qquad R_t|_{t=0} = R_{\nu_0},
\end{equation}
where $L^\dagger$ is related to $L^\sigma$ according to the rule
\begin{equation}
  \label{A4a}
\int_{\Gamma_0} (L^\sigma F)(\eta) R(\eta) \lambda (d \eta) =
\int_{\Gamma_0} F(\eta) (L^\dagger R)(\eta) \lambda (d \eta),
\end{equation}
by which and  \eqref{OperatorLSigma}, \eqref{Minlos1} and
\eqref{Minlos2} one obtains
\begin{eqnarray}
\label{A5}
(L^{\dag}R)(\eta) & = & \frac{1}{2} \suml{z \in \eta} \int\limits_{(\mathds{R}^d)^2} \psi_\sigma(z) c_1(x,y;z)
 R\big(\eta\cup \{x,y \} \bs z \big) dx dy  \\[.2cm] \nonumber & + &
\suml{y \in \eta} \intRd \psi_\sigma(y) c_2(x-y) \prodl{u \in \eta
\bs y} e^{-\phi(y-u)} R(\eta \cup x \bs y) dx \\[.2cm] \nonumber & - & E^\sigma(\eta)
R(\eta),
\end{eqnarray}
where
\begin{eqnarray}
  \label{A5b}
E^\sigma & = & E^\sigma_1 + E^\sigma_2, \\[.2cm] \nonumber
E^\sigma_1(\eta) & = & \suml{\{x,y\} \subset \eta} \  \intRd
\psi_\sigma(z) c_1(x,y;z) dz,
\\[.2cm] \nonumber
 E^\sigma_2(\eta) & = & \suml{x \in \eta} \  \intRd
\psi_\sigma(y) \tilde{c}_2(x;y;\eta) dy.
\end{eqnarray}
Our aim now is to show that $L^\dagger$ is the generator of a
stochastic semigroup $S^\dagger=\{S^\dagger(t)\}_{t\geq 0}$, which
we will use to obtain $R^{\Lambda,N}_t$ in the form $S^\dagger
(t)R^{\Lambda,N}_0$. In doing this, we follow the scheme
developed in \cite[Sect. 3.1]{Kawasaki2}.

The semigroup $S^\dagger$ is supposed to act in the space
$\mathcal{G}_0$, see (\ref{GThetaSpace}). Along with this space we
will also use
\begin{equation}
  \label{A7}
\mathcal{G}_\theta^{fac}  =\{G:\Gamma_0\to \mathds{R}: |G|_{{\rm
fac}, \theta} <\infty\}, \quad |G|_{{\rm fac}, \theta}:=
\int_{\Gamma_0} |G(\eta)| |\eta|! e^{\theta|\eta|} \lambda (d \eta).
\end{equation}
Clearly, for each $\theta\in \mathds{R}$ and $\theta'>\theta$, we
have that
\begin{equation}
  \label{A7a}
\mathcal{G}_\theta^{fac} \hookrightarrow \mathcal{G}_0, \qquad
\mathcal{G}_{\theta'}^{fac} \hookrightarrow
\mathcal{G}_{\theta}^{fac}.
\end{equation}
Note that these are $AL$-spaces, which means that their norms are
additive on the corresponding cones of positive elements
\begin{equation*}
\mathcal{G}^+_0 =\{ G\in \mathcal{G}_0: G(\eta) \geq 0\}, \qquad
\mathcal{G}_\theta^{fac,+} = \{G\in \mathcal{G}_\theta^{fac} :
G(\eta) \geq 0\}.
\end{equation*}
These cones naturally define the cones of positive operators acting
in the corresponding spaces. It is also convenient to relate this
property to the following linear functionals
\begin{equation}
 \label{A9}
\varphi (G) = \int_{\Gamma_0} G(\eta) \lambda (d \eta), \qquad
\varphi^{fac}_\theta (G) = \int_{\Gamma_0} G(\eta)|\eta|!
e^{\theta|\eta|} \lambda (d \eta).
\end{equation}
Then $\varphi(G) =|G|_0$ and $\varphi^{fac}_\theta (G)=|G|_{{\rm
fac},\theta}$ for $G\in \mathcal{G}_0^+$ and $G\in
\mathcal{G}_\theta^{fac,+}$, respectively.

To formulate (\ref{A4}) in the Banach space
$\mathcal{G}_0$, we have to define the corresponding domain of
$L^\dagger$. To this end, we write
\begin{equation}
  \label{A10}
L^\dagger = A^\dagger + B^\dagger, \qquad (A^\dagger G)(\eta) = -
E^\sigma (\eta) G(\eta),
\end{equation}
and then set
\begin{equation}
  \label{A11}
\mathcal{D}^\dagger = \{G\in \mathcal{G}_0: \int_{\Gamma_0} E^\sigma
(\eta) |G(\eta)|\lambda (d \eta) < \infty\}.
\end{equation}
By means of (\ref{Minlos1}) we obtain that
\begin{equation}
  \label{A12}
\int_{\Gamma_0} |(B G)(\eta)| \lambda (d\eta) \leq \int_{\Gamma_0}
E^\sigma (\eta) |G(\eta)|\lambda (d \eta).
\end{equation}
By (\ref{A12}) and the evident positivity of $B$ it follows that
\begin{equation}
  \label{A13}
B :\mathcal{D}^{\dagger}_+ \to \mathcal{G}_0^{+}, \qquad
\mathcal{D}^{\dagger}_+ := \mathcal{D}^{\dagger}\cap
\mathcal{G}_0^+.
\end{equation}
Define
\begin{equation}
  \label{A14}
\mathcal{D}^\dagger_\theta = \{ G \in \mathcal{G}_\theta^{fac}:
A^\dagger G \in \mathcal{G}_\theta^{fac}\}, \qquad A^\dagger_\theta
= A^\dagger|_{\mathcal{D}^\dagger_\theta}.
\end{equation}
That is, $A^\dagger_\theta$ is the restriction of $A^\dagger$ to
$\mathcal{D}^\dagger_\theta$ -- the trace of $A^\dag$ in
$\mathcal{G}_\theta^{fac}$. The construction of the semigroup
$S^\dagger$ is performed by means of the perturbation technique
developed in \cite{ThiemeVoigt}. We formulate here the corresponding
statement borrowed from \cite[Proposition 3.2]{Kawasaki2} in the
form adapted to the present context.
\begin{proposition}
 \label{Apn}
Assume that the operators introduced in (\ref{A10}), (\ref{A11})
have the following properties:
\begin{itemize}
  \item[(i)] $-A^\dagger: \mathcal{D}^{\dagger}_+ \to
\mathcal{G}_0^{+}$ and $B^\dagger: \mathcal{D}^{\dagger}_+ \to
\mathcal{G}_0^{+}$;
\item[(ii)] $(A^\dagger,\mathcal{D}^\dagger)$ is the generator of a sub-stochastic semigroup $S^\dagger_0 =
\{S_0^\dagger(t)\}_{t\geq 0}$ on $\mathcal{G}_0$ such that, for all
$t>0$, $S_0^\dagger(t): \mathcal{G}_\theta^{fac} \to
\mathcal{G}_\theta^{fac}$ and the restrictions $S_0^\dagger(t)|_{
\mathcal{G}_\theta^{fac}}$ constitute a $C_0$-semigroup of operators
on $\mathcal{G}_\theta^{fac}$ generated by $(A^\dagger_\theta,
\mathcal{D}^\dagger_\theta)$ defined in (\ref{A14});
\item[(iii)] $B^\dagger :  \mathcal{D}^\dagger_\theta \to
\mathcal{G}^{fac}_\theta$ and $\varphi ((A^\dagger+B^\dagger)G) =0$
for all $G\in \mathcal{D}^\dagger_+$;
\item[(iv)] there exsist positive $c$ and $\varepsilon$ such that
\[
\varphi_\theta^{fac} \left((A^\dagger + B^\dagger)G \right) \leq c
\varphi_\theta^{fac} (G) - \varepsilon |A^\dagger G|_0 , \quad {\rm
for} \ {\rm all} \  G\in \mathcal{D}^\dagger_\theta \cap
\mathcal{G}_0^{+}.
\]
\end{itemize}
Then the closure of $(A^\dagger, \mathcal{D}^\dagger)$ in
$\mathcal{G}_0$ is the generator of a stochastic semigroup
$S^\dagger = \{S^\dagger(t)\}_{t\geq 0}$ in $\mathcal{G}_0$ that
leaves $\mathcal{G}^{fac}_\theta$ invariant.
\end{proposition}
By means of this statement we prove the following.
\begin{lemma}
  \label{A1lm}
The closure of the operator $(A^\dagger,\mathcal{D}^\dagger)$
defined in (\ref{A10}), (\ref{A11}) generates a stochastic semigroup
$S^\dagger=\{S^\dagger(t)\}_{t\geq 0}$ of operators in
$\mathcal{G}_0$ that leaves invariant $\mathcal{G}_\theta^{fac}$
with any $\theta\in \mathds{R}$.
\end{lemma}
\begin{proof}
 We ought to show that all the conditions of Proposition
\ref{Apn} are met. Condition (i) is met by (\ref{A10}), (\ref{A11})
(case of $A^\dagger$), and by (\ref{A12}), (\ref{A13}) (case of
$B^\dagger$). The operators $S_0^\dag(t)$ mentioned in item (ii) act
as follows $(S_0^\dag(t)G)(\eta) = \exp(- t E^\sigma(\eta))
G(\eta)$. Hence, to prove that all the conditions mentined in (ii)
are  met we have to show the continuity of the map $t \mapsto
(S_0^\dag(t)G)(\eta) \in \mathcal{G}^{fac}_\theta$. That is, we have
to show that
\[
\int_{\Gamma_0} \bigg{(} 1 - \exp\left(-t E^\sigma (\eta)
\right)\bigg{)} |G(\eta)| e^{\theta|\eta|} |\eta|! \lambda (d \eta)
\to 0, \qquad {\rm as} \ t \to 0^{+},
\]
which obviously holds by the Lebesgue dominated convergence theorem
since $G\in \mathcal{G}_\theta^{fac}$. To check that
$B^\dagger:\mathcal{D}_\theta^{\dag} \to \mathcal{G}_\theta^{fac}$,
for $G\in \mathcal{D}^\dag_{\theta}\cap \mathcal{G}_0^+$, we apply
(\ref{Minlos1}) and obtain from (\ref{A5}) that
\begin{eqnarray}
\label{A15} & & \varphi_\theta^{fac}(B^\dag G) \\[.2cm] \nonumber & & \quad  = \frac{1}{2}
\int_{\Gamma_0} |\eta|! e^{\theta|\eta|} \bigg{(}\sum_{z\in \eta}
\int_{(\mathds{R}^d)^2} \psi_\sigma (z) c_1 (x,y;z)
G(\eta\cup\{x,y\}\setminus z) d x d y \bigg{)} \lambda (d \eta)
\qquad  \\[.2cm] \nonumber & &  \quad + \int_{\Gamma_0} |\eta|! e^{\theta|\eta|} \bigg{(}\sum_{y\in \eta}
\int_{\mathds{R}^d} \psi_\sigma (y) \tilde{c}_2 (x,y;\eta)
G(\eta\cup x \setminus y) d x  \bigg{)} \lambda (d \eta)
\qquad  \\[.2cm] \nonumber & & \quad = e^{-\theta} \int_{\Gamma_0} (|\eta|-1)! e^{\theta|\eta|}
\bigg{(} \sum_{\{x,y\}\subset \eta} \int_{\mathds{R}^d} \psi_\sigma
(z) c_1 (x,y;z) d z \bigg{)} G(\eta) \lambda ( d \eta) \qquad  \\[.2cm] \nonumber & &
\quad + \int_{\Gamma_0} |\eta|! e^{\theta|\eta|} \bigg{(}\sum_{x\in
\eta} \int_{\mathds{R}^d} \psi_\sigma (y) \tilde{c}_2 (x,y;\eta) dy
 \bigg{)} G(\eta)   \lambda (d \eta) \qquad  \\[.2cm] \nonumber & &
\quad \leq \max\{1; e^{-\theta}\} \int_{\Gamma_0} |\eta|!
e^{\theta|\eta|} E^\sigma (\eta) G(\eta) \lambda ( d \eta).
\end{eqnarray}
This yields $B^\dagger:\mathcal{D}_\theta^{\dag} \to
\mathcal{G}_\theta^{fac}$. In the same way, one shows that
\[
\varphi(B^\dagger G) = \int_{\Gamma_0} E^\sigma (\eta) G(\eta)
\lambda ( d \eta),
\]
that completes the proof of item (iii).

To prove that (iv) holds, by (\ref{A10}), (\ref{A5b}) and
(\ref{A15}) we obtain
\[
\varphi_\theta^{fac} \left((A^\dagger + B^\dagger)G \right) =
\int_{\Gamma_0} (|\eta|-1)! e^{\theta|\eta|} G(\eta) E_1^\sigma
(\eta) \left(e^{-\theta} -|\eta| \right)\lambda ( d\eta).
\]
Then the inequality in item (iv) can be written in the form
\[
- c G(\emptyset) + \int_{\Gamma_0}|\eta|! e^{\theta|\eta|}
W(c,\varepsilon;\eta) G(\eta) \lambda (d \eta) \leq 0,
\]
where $W(c,\varepsilon;\emptyset) = 0$ and
\begin{eqnarray}
  \label{A16}
W(c,\varepsilon;\eta)=  \left( \frac{e^{-\theta}}{|\eta|} -1 \right)
E^\sigma_1 (\eta) + \varepsilon \frac{e^{-\theta|\eta|}}{|\eta|!}
E^\sigma (\eta)  - c, \qquad |\eta|\geq 1 .
\end{eqnarray}
Since both first and second summands in (\ref{A16}) are bounded from above in
$\eta$, one can pick $c>0$ big enough to make
$W(c,\varepsilon;\eta)\leq 0$ for all $\eta\in \Gamma_0$. This
completes the  proof of the lemma.
\end{proof}
By (\ref{DensExt}) and (\ref{A9}) we have that
\[
|R_0^{\Lambda,N}|_{{\rm fac},\theta} = \varphi_\theta^{fac}
(R_0^{\Lambda,N}) \leq N! e^{\theta N} \varphi (R^\Lambda_{\mu_0}) =
N! e^{\theta N}.
\]
Hence, $R_0^{\Lambda,N} \in \mathcal{G}_\theta^{fac}$ for any
$\theta\in \mathds{R}$. By the same arguments we also have that
$|R_0^{\Lambda,N}|_0 \leq 1$. Then, for all $t>0$, we have that
\begin{equation}
  \label{A17}
R^{\Lambda,N}_t = S^\dag (t) R_0^{\Lambda,N} \in
\mathcal{G}_\theta^{fac} \cap \mathcal{G}_0^{+},
\end{equation}
and $|R^{\Lambda,N}_t|_0\leq 1$. Also, $R^{\Lambda,N}_t$ is a unique
classical solution of the problem in (\ref{A4}) with the initial
condition $R^{\Lambda,N}_0$. These facts follow by Lemma \ref{A1lm}.
Now we define $q^{\Lambda,N}_t$ by (\ref{LocalCorFun}), see also
(\ref{CorAndDens}). Then $q^{\Lambda,N}_t(\emptyset) =
|R^{\Lambda,N}_t|_0\leq 1$ and $q^{\Lambda,N}_t$ has  both (b) and
(c) properties of Proposition \ref{CorFunProp}. By (\ref{A4a}) and
then by (\ref{o9}) and (\ref{o2}) we conclude that
\begin{equation}
  \label{A18}
\langle\! \langle F^\omega, L^\dag F^{\Lambda,N}_t \rangle \!\rangle
= \langle\! \langle e(\omega;\cdot), L^{\Delta,\sigma}
q^{\Lambda,N}_t \rangle \!\rangle.
\end{equation}
The latter means that one can apply $L^{\Delta,\sigma}$ to
$q^{\Lambda,N}_t$ pointwise, and then calculate the integral with
$e(\omega;\cdot)$. At the same time, for each $\theta\in \mathds{R}$,
we have
\begin{gather}\label{QInGFac}
|q_t^{\Lambda, N}|_{\text{fac}, \theta} = \intGamN R_t^{\Lambda, N}
(\eta) \suml{\xi \subset \eta} e^{\theta |\xi|} |\xi|! \lambda(d\eta)\\[.2cm]
 = \intGamN R_t^{\Lambda, N} (\eta) \suml{k = 0}^{|\eta|} \frac{|\eta|!}{(|\eta| - k)!} e^{\theta k} \lambda(d\eta)
\leq e^{e^{- \theta}} |R_t^{\Lambda, N}|_{\text{fac}, \theta}.
\nonumber
\end{gather}
Keeping in mind (\ref{A18}) and (\ref{QInGFac}) let us define
$L^{\Delta,\sigma}$ in a given $\mathcal{G}^{fac}_\theta$. To this
end, we set, cf. (\ref{hD}),
\begin{equation}
  \label{A19}
\mathcal{D}^{fac}_\theta =\{ G \in \mathcal{G}_\theta^{fac}:
|\cdot|^2 G \in \mathcal{G}_\theta^{fac}\},
\end{equation}
and, see (\ref{S11}),
\begin{equation}
  \label{A20}
(A^{fac}G)(\theta) = - \Psi_\sigma (\eta) G(\eta), \qquad B^{fac} =
L^{\Delta,\sigma}- A^{fac}.
\end{equation}
As above, one shows that both $A^{fac}$ and $B^{fac}$ map
$\mathcal{D}^{fac}_\theta$ into $\mathcal{G}_\theta^{fac}$ that
allows for defining the corresponding unbounded operators
$A^{fac}_\theta = (A^{fac}, \mathcal{D}^{fac}_\theta)$,
$B^{fac}_\theta = (B^{fac}, \mathcal{D}^{fac}_\theta)$ and
$L^{{fac},\Delta,\sigma}_{\theta} = (A^{fac}+B^{fac},
\mathcal{D}^{fac}_\theta)$. Similarly as in the proof of Theorem
\ref{Th1} we can define the corresponding bounded operators acting
from $\mathcal{G}_{\theta'}^{fac}$ to $\mathcal{G}_{\theta}^{fac}$
for $\theta'>\theta$, see (\ref{A7a}). Their operator norms satisfy
\begin{equation}
\label{A19a}
    \|A^{fac}_{\theta \theta'}\| \leq \frac{2 \coMax }{(\theta' - \theta)^2
e^2}, \qquad
    \|B^{fac}_{\theta \theta'}\| \leq \frac{\frac{3}{2} \coMax e^{-\theta } + 2 e^{ e^{- \theta  }} \cTwo}{(\theta'  - \theta
)e}.
\end{equation}
By (\ref{A7a}) and (\ref{A19}) we also have that
\begin{equation*}
\forall \theta'> \theta \qquad \mathcal{G}^{fac}_{\theta'} \subset
\mathcal{D}^{fac}_{\theta}.
\end{equation*}
\begin{lemma}
  \label{A2lm}
For each $\theta\in \mathds{R}$, the function $t\mapsto
q^{\Lambda,N}_t\in \mathcal{D}^{fac}_{\theta} \subset
\mathcal{G}^{fac}_{\theta}$ defined in (\ref{LocalCorFun}) is a
unique global in time classical solution of the Cauchy problem
\begin{equation}
  \label{A21}
\frac{d}{dt}G_t = L^{{fac},\Delta,\sigma}_{\theta} G_t, \qquad
G_t|_{t=0} = q^{\Lambda,N}_0,
\end{equation}
with $R^{\Lambda,N}_0$ defined in (\ref{DensExt}).
\end{lemma}
\begin{proof}
The continuity and continuous differentiability of the function
$t\mapsto q^{\Lambda,N}_t\in \mathcal{G}^{fac}_{\theta}$ follow by
(\ref{A17}) and the fact that $S^\dag$ is a $C_0$-semigroup. The
inclusion $q^{\Lambda,N}_t\in \mathcal{D}^{fac}_{\theta}$ follows by
(\ref{QInGFac}) and the fact that $R^{\Lambda,N}_t\in
\mathcal{G}^{fac}_{\theta'}$ with an arbitrary $\theta'$, see
(\ref{A7a}) and (\ref{A20}). The fact that $q^{\Lambda,N}_t$
satisfies the first equality in (\ref{A21}) can be proved by means
of (\ref{A18}). Thus, it remains to prove the stated uniqueness.
Take any $\theta''>\theta'>\theta$ and consider the problem in
(\ref{A21}) in $\mathcal{G}^{fac}_{\theta'}$. Since the initial
condition $q_0^{\Lambda,N}$ lies in $\mathcal{G}^{fac}_{\theta''}$,
by means of the estimates in (\ref{A19a}) and the technique
developed for the proof of Theorem \ref{Th1} one can prove the
existence of a unique classical solution of the latter problem in
$\mathcal{G}^{fac}_{\theta'}$, on a bounded time interval. The
latter thus coincides with the one given in (\ref{LocalCorFun}),
which yields the uniqueness in question of this interval. Its
further continuation is performed by repeating the same arguments.
\end{proof}

\subsubsection{The common evolution}\label{Subsection:Spaces}

Our aim now is to show that (\ref{A3}) holds, which by
(\ref{LocalCorFun}) would yield the desired positivity of
$k_t^{\Lambda,N}$. A priori (\ref{A3}) does not make any sense as
$k_t^{\Lambda,N}$ and $q_t^{\Lambda,N}$ belong to different spaces
(and are not defined pointwise). The resolution consists in placing
$k_t^{\Lambda,N}$ and $q_t^{\Lambda,N}$ into a common subspace of
these two spaces, that is $\Usigma_\theta$ which we define now.

For $u:\Gamma_0\to \mathds{R}$, we set
\begin{equation*}
\|u\|_{\sigma, \theta} = \esssup_{\eta \in \Gamma_0} \frac{|u(\eta)|
e^{- \theta |\eta|}}{e(\psi_{\sigma};\eta)},
\end{equation*}
where $e(\psi_{\sigma};\eta)$ = $\prodl{x \in \eta} e^{- \sigma
|x|^2}$. Analogously as in \eqref{KNormEstimate}, we have
\begin{equation}\label{UNormEstimate}
    |u(\eta)| \leq  e^{\theta |\eta|} e(\psi_\sigma; \eta) \|u\|_{\sigma, \theta}.
\end{equation}
Then
\begin{equation}
  \label{A23}
\Usigma_\theta:=\{ u:\Gamma_0\to \mathds{R}: \|u\|_{\sigma,
\theta}<\infty\} \subset \mathcal{K}_{\theta}.
\end{equation}
By (\ref{UNormEstimate}) and (\ref{A7}), for $u \in \Usigma_\theta$,
we have that
\begin{eqnarray*}
    |u|_{\text{fac}, \theta'} & \leq & \intGamN e(\psi_\sigma; \eta) e^{(\theta + \theta') |\eta|} |\eta|! ||u||_{\sigma, \theta}
\lambda(d\eta)\\[.2cm] \nonumber
   &  = & \|u\|_{\sigma,
\theta} \sum\limits_{n=0}^{\infty} \left(e^{\theta + \theta'}
\intRd \psi_\sigma(x) dx \right)^n = \|u\|_{\sigma, \theta}
\sum\limits_{n=0}^{\infty} \left( e^{\theta + \theta'}
\left(\frac{\pi}{\sigma}\right)^{d/2}  \right)^n,
\end{eqnarray*}
which yields, cf. (\ref{A23}),
\begin{equation}\label{USubsetG}
    \mathcal{U}^\sigma_\theta \subset \mathcal{G}^{\text{fac}}_{\theta'}, \qquad {\rm for} \ \ \theta' < -\theta - \frac{d}{2}(\ln \pi - \ln \sigma).
\end{equation}
Since $k_0^{\Lambda, N} = q_0^{\Lambda, N}$, it belongs to $
\mathcal{U}^\sigma_{\alpha_0}$ and 
to $\mathcal{G}^{\text{fac}}_{\beta_0}$ with
\begin{equation}\label{BetaToAlphaRelation}
    \beta_0 < -\alpha_0 - \frac{d}{2}(\ln \pi - \ln \sigma).
\end{equation}
To prove the former, by \eqref{kLambdaN} we readily get
\[
   \|k_0^{\Lambda, N}\|_{\sigma, \theta} \leq \exp\left( \sigma N \sup_{y\in \Lambda} |y|^2 \right)  \|k_{\mu_0}\|_{\alpha_0}.
\]
Our aim now is to prove that both evolutions $q_0^{\Lambda,N} =
k_0^{\Lambda,N} \to k_t^{\Lambda,N}$ and $q_0^{\Lambda,N}  \to
q_t^{\Lambda,N}$ take place in $\mathcal{U}^\sigma_{\alpha_*}$. To this end, we define
$L^{\Delta,\sigma}$ in $\mathcal{U}^\sigma_{\theta}$ and split it
$L^{\Delta,\sigma}= A^\sigma+B^\sigma$, as we did in (\ref{S11}).
Then set
\begin{equation}
  \label{A24}
\mathcal{D}^\sigma_\theta =\{u \in \mathcal{U}^\sigma_\theta:
|\cdot|^2 u\in \mathcal{U}^\sigma_\theta\}\subset
\mathcal{D}_\theta,
\end{equation}
see (\ref{o10}) and (\ref{A23}). At the same time, similarly as in (\ref{USubsetG}) one shows that
\begin{equation}
\label{A24a} \mathcal{D}_\theta^\sigma \subset
\mathcal{D}^{fac}_{\theta'}, \quad {\rm for} \ \ \theta' < -\theta
-\frac{d}{2} (\ln \pi - \ln \sigma).
\end{equation}
Like above, one can show that $\mathcal{U}^\sigma_{\theta''}\subset
\mathcal{D}^\sigma_\theta$ whenever $\theta''<\theta$, cf.
(\ref{o11}). Thus, $A^\sigma: \mathcal{D}^\sigma_\theta \to
\mathcal{U}^\sigma_\theta$. Likewise, $B^\sigma:
\mathcal{D}^\sigma_\theta \to \mathcal{U}^\sigma_\theta$, and hence
we can define in $\mathcal{U}^\sigma_\theta$ the unbounded operators
$A^\sigma_{u,\theta} = (A^\sigma,\mathcal{D}^\sigma_\theta)$,
$B^\sigma_{u,\theta} = (B^\sigma,\mathcal{D}^\sigma_\theta)$ and
$L^{\Delta, \sigma}_{u,\theta} =
(A^\sigma+B^\sigma,\mathcal{D}^\sigma_\theta)$. Note that
$L^{\Delta, \sigma}_{u,\theta}$ satisfies, cf. (\ref{A24}),
\begin{equation}
  \label{A25a}
\forall u \in \mathcal{D}^\sigma_\theta \qquad L^{\Delta,
\sigma}_{u,\theta} u = L^{\Delta, \sigma}_{\theta} u,
\end{equation}
where the latter operator is the same as in (\ref{SigmaEv}).
Likewise, by (\ref{A24a}) we have
\begin{equation}
\label{A25b}
\forall u \in \mathcal{D}^\sigma_\theta \qquad L^{\Delta,
    \sigma}_{u,\theta} u = L^{fac,\Delta,\sigma}_{\theta'} u,
\end{equation}
with $\theta'$ satisfying the bound in (\ref{A24a}). That is,
$L^{\Delta, \sigma}_{u,\theta}= L^{\Delta,
\sigma}_{\theta}|_{\mathcal{D}_\theta^\sigma}$, that holds for each
$\theta\in \mathds{R}$; as well as, $L^{\Delta, \sigma}_{u,\theta}=
L^{fac,\Delta,\sigma}_{\theta'}|_{\mathcal{D}_\theta^\sigma}$,
holding for all $\theta$ and $\theta'$ satisfying (\ref{USubsetG}).

Let us now consider the problem
\begin{equation}
  \label{A25}
\frac{d}{dt} u_t = L^{\Delta, \sigma}_{u,\alpha_*} u_t, \qquad
u_t|_{t=0} = q_0^{\Lambda,N}.
\end{equation}
Its solution is to be understood according to Definition \ref{Odf}.
\begin{lemma}
  \label{A3lm}
Let $\alpha_*$ and $\alpha_0$ be as in Theorem \ref{Th1}, and then
$T(\alpha_*, \alpha_0)$ be as in (\ref{o26}). Then the problem in
(\ref{A25}) has a unique classical solution in
$\mathcal{U}^\sigma_{\alpha_*}$ on the time interval
$[0,T(\alpha_*,\alpha_0))$.
\end{lemma}
\begin{proof}
As in the case of Theorem \ref{Th1}, the present proof is based on
the following estimates of the summands of $L^{\Delta,\sigma}$, cf.
(\ref{O2}) and (\ref{O3}). By \eqref{UNormEstimate} and
\eqref{Minlos1} together with \eqref{Minlos2}, for $\theta' >
\theta$, we get
\begin{align}
\label{A26}
 \|L_{11}^{\Delta, \sigma} u\|_{\sigma, \theta'}
&\leq \esssup_{\eta \in \Gamma_0} \frac{e^{- \theta'|\eta|}}{e(\psi_\sigma; \eta)}
 \cdot \frac{1}{2} \intl{(\mathds{R}^d)^2} \suml{z \in \eta} \psi_\sigma(z) c_1(x,y;z) |u(\eta \bs z \cup \{x,y\})| dx dy
\\[.2cm] \nonumber
   &\leq \esssup_{\eta \in \Gamma_0} e^{- (\theta' - \theta)|\eta|} \cdot \frac{1}{2} e^\theta \|u\|_{\sigma, \theta}
 \intl{(\mathds{R}^d)^2} \suml{z \in \eta} c_1(x,y;z) \psi_\sigma(x) \psi_\sigma(y) dx dy
\\[.2cm] \nonumber
    &\leq  \frac{e^\theta \cOne}{2 e (\theta' - \theta)} \|u\|_{\sigma,
\theta}.
\end{align}
Similarly, one obtains, cf. (\ref{O2}), that $L_{12}^{\Delta,
\sigma}$ and $L_{13}^{\Delta, \sigma}$ satisfy (\ref{A26}), and
furthermore,  cf. (\ref{O4}),
\begin{eqnarray*}
\|L_{2}^{\Delta, \sigma} u\|_{\sigma, \theta'} \leq \frac{2 \langle
c_2 \rangle}{e(\theta'-\theta)}\exp\left(\langle \phi \rangle
e^\theta \right)  \|u\|_{\sigma}.
\end{eqnarray*}
By means of these estimates we define a bounded operator
$(B^\sigma_u)_{\theta'\theta}$ acting from
$\mathcal{U}^\sigma_\theta$ to $\mathcal{U}^\sigma_{\theta'}$. Its
norm satisfies the corresponding estimate in (\ref{O10}) with the
same right-hand side. Then the proof follows in the same way as in
the case of Theorem \ref{Th1}.
\end{proof}
\begin{corollary}
  \label{Aco}
For each fixed $\sigma>0$ and all $t<T(\alpha_*, \alpha_0)$, it
follows that $k^{\Lambda,N}_t=q^{\Lambda,N}_t$, and hence
$k^{\Lambda,N}_t$ satisfies (\ref{A2}) for all $G\in
\mathcal{B}_{bs}^*$.
\end{corollary}
\begin{proof}
We take
\begin{equation}
 \label{A27a}
 \beta_* < - \alpha_* - \frac{d}{2}(\ln \pi - \ln \sigma),
\end{equation}
and obtain by (\ref{USubsetG}), (\ref{BetaToAlphaRelation}) and
(\ref{A23}) that
\begin{equation}
  \label{A28}
q_0^{\Lambda,N} = k_0^{\Lambda,N} \in
\mathcal{U}^\sigma_{\alpha_0} \subset \mathcal{G}^{fac}_{\beta_0}
\cap \mathcal{K}_{\alpha_0}, \qquad
 \mathcal{U}^\sigma_{\alpha_*} \subset \mathcal{G}^{fac}_{\beta_*}
\cap \mathcal{K}_{\alpha_*}.
\end{equation}
Then $k_t^{\Lambda,N} =
Q^\sigma_{\alpha_*\alpha_0}(t) k_0^{\Lambda,N}$ is a unique
classical solution of the problem in (\ref{SigmaEv}) with the
initial condition $k_0^{\Lambda,N}$. Let $u_t$,
$t<T(\alpha_*,\alpha_0)$ be the solution of the problem in
(\ref{A25}). Then the map $t\mapsto u_t\in
\mathcal{U}^\sigma_{\alpha_*} \subset \mathcal{K}_{\alpha_*}$, cf.
(\ref{A28}), is continuous and continuously differentiable in
$\mathcal{K}_{\alpha_*}$ as the corresponding embedding is
continuous. By (\ref{A25a}) $u_t$ satisfies also (\ref{SigmaEv})
with the initial condition $q_0^{\Lambda,N}=k_0^{\Lambda,N}$, and
hence $u_t = k^{\Lambda,N}_t =
Q^\sigma_{\alpha_*\alpha_0}(t)q_0^{\Lambda,N}$ in view of the
uniqueness of the solution of (\ref{SigmaEv}). Likewise, by
(\ref{A25b}) and Lemma \ref{A2lm} one proves that $u_t =
q_t^{\Lambda,N}$, $t<T(\alpha_*,\alpha_0)$, where $q_t^{\Lambda,N}$
is the unique solution of (\ref{A21}) in
$\mathcal{G}^{fac}_{\beta_*}$ with $\beta_*$ as in (\ref{A27a}).
\end{proof}

\subsubsection{Eliminating the approximations}\label{Subsection:LambdaNConv}

We recall that $k^{\Lambda,N}_t = Q^\sigma_{\alpha_*\alpha_0}(t)
k_0^{\Lambda,N}$ approximates $k^{\sigma}_t =
Q^\sigma_{\alpha_*\alpha_0}(t) k_0$ that solves (\ref{SigmaEv}) and
is mentioned in Lemma \ref{S1lm}. Our aim now is to eliminate this
approximation. By a cofinal sequence of compact subsets
$\{\Lambda_n\}_{n\in \mathds{N}}$ we mean a sequence which is ordered
by inclusion $\Lambda_n \subset \Lambda_{n+1}$ and exhaustive in the
sense that each $x\in \mathds{R}^d$ is eventually contained in its
element.
\begin{lemma}
  \label{A5lm}
 For each fixed $\sigma>0$, $t<T(\alpha_*, \alpha_0)$ and any
 cofinal sequence $\{\Lambda_n\}_{n\in \mathds{N}}$, it follows that
\begin{equation}\label{ConvNLamb}
\forall G\in   \bbs\qquad     \lim_{n \rightarrow \infty}\left(
\lim_{N \rightarrow \infty}\langle \!\langle G, k_t^{\Lambda_n,
N}\rangle \!\rangle\right) = \langle \!\langle G,
k_t^{\sigma}\rangle \!\rangle.
\end{equation}
\end{lemma}
\begin{proof}
We mostly follow the line of arguments used in
\cite[Appendix]{Kawasaki}. Throughout the proof $\sigma>0$ will be
fixed. For $t<T(\alpha_*,\alpha_0)$, by \eqref{QHDUality} we have
\[
\langle\!\langle G, k_t^{\Lambda_n, N}\rangle \!\rangle =
\langle\!\langle H_{\alpha_* \alpha_0}(t) G , k_0^{\Lambda_n, N}
\rangle \!\rangle , \quad
 \langle\!\langle G, k_t^{\sigma} \rangle \!\rangle = \langle\!\langle H_{\alpha_* \alpha_0}(t)G, k_0)\rangle \!\rangle.
\]
Set $G_t = H_{\alpha_* \alpha_0}(t) G$, and then write
\begin{equation*}
  \delta(n,N):= \langle\!\langle G, k_t^{\sigma} \rangle \!\rangle -\langle\!\langle G, k_t^{\Lambda_n, N}\rangle
  \!\rangle = \langle\!\langle G_t, k_0 - k_0^{\Lambda_n, N} \rangle
  \!\rangle =:J^{(1)}_n + J^{(2)}_{n,N}.
\end{equation*}
Here
\begin{eqnarray}
  \label{A36}
J^{(1)}_n &= & \intGamN G_t(\eta) k_0(\eta) \left(1 -
\mathds{1}_{\Gamma_{\Lambda_n}}(\eta)\right) \lambda(d\eta),
    \\[.2cm] \nonumber
    J^{(2)}_{n,N} &= & \intGamN G_t(\eta) \left( k_0(\eta) \mathds{1}_{\Gamma_{\Lambda_n}}(\eta) - k_0^{\Lambda_n, N}(\eta) \right)
    \lambda(d\eta),
\end{eqnarray}
and $\mathds{1}_{\Gamma_{\Lambda_n}}$ is the indicator of
$\Gamma_{\Lambda_n}$. Let us prove that, for each $n$, $J^{(2)}_{n,N}\to 0$
as $N\to +\infty$. To this end we rewrite it in the following form
\begin{eqnarray}
\label{A40} J^{(2)}_{n,N} &= & \intGamN \Big[ G_t(\eta)
\intl{\Gamma_{\Lambda_n}} R_0^{\Lambda_n}(\eta \cup \xi)
I_{\Gamma_{\Lambda_n}}(\eta) \left( 1 -  I_N(\eta \cup \xi) \right)
\lambda(d\xi) \Big] \lambda(d\eta)
    \\[.2cm] \nonumber
&= & \intGamN G_t(\eta) \intGamN R_0^{\Lambda_n}(\eta \cup \xi)
I_{\Gamma_{\Lambda_n}}(\eta \cup \xi)\left(1 - I_N(\eta \cup \xi)
\right) \lambda(d\xi) \lambda(d\eta)
    \\[.2cm] \nonumber
&=& \intl{\Gamma_{\Lambda_n}} \suml{\xi \subset \eta} G_t(\xi)
R_0^{\Lambda_n}(\eta) \left( 1 - I_N(\eta) \right) \lambda(d\eta)
    \\[.2cm] \nonumber
&= & \suml{m = N+1}^\infty \frac{1}{m!} \intl{(\Lambda_n)^m}
R_0^{\Lambda_n}(\{x_1, \ldots, x_m\})\\[.2cm] \nonumber & \times & \suml{k = 0}^m \ \suml{\{i_1,
\ldots i_k \} \subset \{1, \ldots m \}} G_t^{(k)}(x_{i_1}, \ldots,
x_{i_k}) dx_1 \cdots dx_m,
\end{eqnarray}
where $I_N(\eta)=1$ whenever $|\eta|\leq N$ and $I_N(\eta)=0$
otherwise. By (\ref{CorAndDens}) for $k_0 \in
\mathcal{K}_{\alpha_0}$, it follows that
\[
R_0^\Lambda(\{y_1, \ldots, y_s\}) \leq k_0(\{y_1, \ldots, y_s\})
\leq e^{\alpha_0 s} \|k_0\|_{\alpha_0}.
\]
We apply this estimate in (\ref{A40}) and obtain
\begin{eqnarray}
\label{A41} 
|J^{(2)}_{n,N}| & \leq & \|k_0\|_{\alpha_0} \\[.2cm]\nonumber  & \times & \suml{m =
N+1}^\infty \frac{1}{m!} e^{\alpha_0 m} \intl{(\Lambda_n)^m} \suml{k
= 0}^m \ \suml{\{i_1, \ldots i_k \} \subset \{1, \ldots m \}}
|G_t^{(k)}(x_{i_1}, \ldots, x_{i_k})| dx_1 \cdots dx_m
    \\[.2cm] \nonumber
& \leq & \|k_0\|_{\alpha_0} \suml{m = N+1}^\infty \frac{1}{m!}
e^{\alpha_0 m} \suml{k = 0}^m \frac{m!}{k! (m-k)!}
\|G_t^{(k)}\|_{L^1\left((\Rd)^k\right)} |\Lambda_n|^{m-k},
\end{eqnarray}
where $|\Lambda|$ stands for the Lebesgue measure of $\Lambda$. The
sum over $m$ in the last line of (\ref{A41}) is the remainder of the
series
\begin{eqnarray*}
& & \suml{m = 0}^\infty \suml{k = 0}^m  \frac{e^{\alpha_0 k}}{k! }
\|G_t^{(k)}\|_{L^1\left((\Rd)^k\right)} \frac{e^{\alpha_0 (m -
k)}}{(m-k)!} |\Lambda_n|^{m-k}
    \\[.2cm]
 & & \quad    = \suml{k = 0}^\infty  \frac{e^{\alpha_0 k}}{k! } \|G_t^{(k)}\|_{L^1\left((\Rd)^k\right)}
  \suml{m = 0}^\infty  \frac{e^{\alpha_0 m}}{m!} |\Lambda_n|^{m}
    \\[.2cm] & &  \quad
    = |G_t|_{\alpha_0} \exp\Big( e^{\alpha_0} |\Lambda_n| \Big).
\end{eqnarray*}
Hence, by (\ref{A41}) we obtain that
\begin{equation*}
  \delta_n := \lim_{N\to +\infty} \delta (n,N) = J_n^{(1)}.
\end{equation*}
Then to complete the proof of (\ref{ConvNLamb}) we should show that
$\delta_n\to 0$. By (\ref{A36}) we have
\[
|J^{(1)}_n| \leq \suml{p=1}^\infty \frac{1}{p!}
\int\limits_{(\Rd)^p} |G_t^{(p)} (x_1, \ldots, x_p)| k_0^{(p)}(x_1,
\ldots, x_p) \suml{l=1}^p \mathds{1}_{\Lambda_N^c}(x_l) dx_1 \cdots
dx_p.
\]
Since $k_0 \in \mathcal{K}_{\alpha_0}$ and $G_t^{(p)}$ and $k_0^{(p)}$
are symmetric, we may rewrite the above estimate as follows
\[
|J^{(1)}_n| \leq \|k_0\|_{\alpha_0} \suml{p=1}^\infty \frac{p}{p!}
e^{\alpha_0 p} \int\limits_{\Lambda_n^c}\int\limits_{(\Rd)^{p-1}}
|G_t^{(p)} (x_1, \ldots, x_p)| dx_1 \cdots dx_p.
\]
For each $t<T(\alpha_*,\alpha_0)$, one finds $\epsilon>0$ such that
$t<T(\alpha_*+\epsilon,\alpha_0+\epsilon)$, see (\ref{o26}). We fix
these $t$ and $\epsilon$. Since $G_0 \in \bbs$, and hence $G_0 \in
\mathcal{G}_{\alpha_*+\epsilon}$, by Lemma \ref{Slm} we have that
$G_t \in \mathcal{G}_{\alpha_0+\epsilon}$. We apply the latter fact in the
estimate above and obtain
\begin{eqnarray}
  \label{A37}
|J^{(1)}_n|  & \leq &  \frac{\|k_0\|_{\alpha_0}}{e \epsilon }
\Delta_n,\\[.2cm] \nonumber
\Delta_n & := & \suml{p=1}^\infty \frac{1}{p!} e^{(\alpha_0+\epsilon)
p} \int\limits_{\Lambda_n^c}\int\limits_{(\Rd)^{p-1}} |G_t^{(p)}
(x_1, \ldots, x_p)| dx_1 \cdots dx_p.
\end{eqnarray}
Furthermore, for each $M\in \mathds{N}$, we have
\begin{gather}
  \label{A51}
  \Delta_n \leq \Delta_{n,M}^{(1)} + \Delta_{M}^{(2)},\\[.2cm]
  \nonumber
 \Delta_{n,M}^{(1)} :=  \suml{p=1}^M \frac{1}{p!} e^{(\alpha_0+\epsilon)
p} \int\limits_{\Lambda_n^c}\int\limits_{(\Rd)^{p-1}} |G_t^{(p)}
(x_1, \ldots, x_p)| dx_1 \cdots dx_p, \\[.2cm] \nonumber
\Delta_{M}^{(2)} := \suml{p=M+1}^\infty \frac{1}{p!}
e^{(\alpha_0+\epsilon) p} \int\limits_{(\Rd)^{p}} |G_t^{(p)} (x_1,
\ldots, x_p)| dx_1 \cdots dx_p,
\end{gather}
Fix some $\varepsilon>0$, and then pick $M$ such that
$\Delta^{(2)}_M < \varepsilon/2$, which is possible since
\[
\suml{p=1}^\infty \frac{1}{p!} e^{(\alpha_0+\epsilon) p}
\int\limits_{(\Rd)^{p}} |G_t^{(p)} (x_1, \ldots, x_p)| dx_1 \cdots
dx_p = |G_t|_{\alpha_0+\epsilon},
\]
as $G_t \in \mathcal{G}_{\alpha_0+\epsilon}$. At the same time,
\[
\forall p\in \mathds{N} \qquad g_p (x) :=
\int_{(\mathds{R}^d)^{p-1}}|G^{(p)}_t (x, x_2, \dots , x_p)|d x_2
\cdots dx_p \in L^1(\mathds{R}^d).
\]
Thus, since the sequence $\{\Lambda_n\}$ is exhausting, for $M$
satisfying $\Delta^{(2)}_M < \varepsilon/2$, there exists $n_1$ such
that, for $n
> n_1$, the following holds
\[
\Delta_{n,M}^{(1)}=  \suml{p=1}^M\frac{1}{p!} e^{(\alpha_0+\epsilon)
p} \int\limits_{\Lambda_n^c} g_p(x) dx < \frac{\varepsilon}{2}.
\]
By (\ref{A51}) this yields $\Delta_n<\varepsilon$ for all $n>n_1$,
which by (\ref{A37}) completes the proof.

\end{proof}
\subsubsection{The proof of Theorem \ref{Th2}} 

The proof will be done by showing that the solution $k_t$ described in Theorem \ref{Th1} has the three properties
mentioned in Proposition \ref{CorFunProp}. 
As discussed at the beginning of Section \ref{Section:ProofTh2}, 
$k_t$ surely has properties (a) and (b). By Corollary \ref{Aco} and Lemma \ref{A5lm} we have that $k_t^\sigma$ satisfies (\ref{o7})
for all $t<T(\alpha_*, \alpha_0)$ and $\sigma>0$. Then by Lemma \ref{S1lm} we get that $k_t$ also satisfies (\ref{o7}) for $t<T(\alpha_*, \alpha_0)/2$, that completes the proof. 

\paragraph{\bf
Acknowledgement} In 2016 and 2017, the research of both authors
under the present paper was supported in part by the European
Commission under the project STREVCOMS PIRSES-2013-612669. In 2018,
Yuri Kozitsky was supported by National Science Centre, Poland,
grant 2017/25/B/ST1/00051. All these supports are cordially
acknowledged by the authors.

\bibliographystyle{plain}

\end{document}